\documentclass{article}
\usepackage{amsfonts}
\usepackage{amssymb}
\usepackage{amsmath}
\usepackage{amsthm}
\usepackage{cite}
\usepackage{CJK}

\newcommand{\Case}[1]{\textbf{Case #1.}}
\numberwithin{equation}{section}

\newtheorem{theorem}{Theorem}[section]
\newtheorem{lemma}{Lemma}[section]
\newtheorem{remark}{Remark}[section]
\newtheorem{example}{Example}[section]

\date{}

\begin{document}

\title {Explicit meromorphic solutions of certain types of nonlinear differential-difference equations}

\author{Min Feng CHEN
\thanks{Corresponding author.
This research was supported by the National Natural
Science Foundation of  China (No. 12001117).},
and Zong Xuan CHEN\\
}\maketitle
\date{}

\begin{minipage}{128mm}

\begin{flushleft}
\quad
\end{flushleft}

{\bf Abstract}\quad
Our main purpose is to determine all meromorphic solutions, in explicit form, of the
nonlinear differential-difference equations
\begin{equation*}
f^{n}(z)f'(z)+q(z)f(z+c)=p_{1}{\rm e}^{\alpha_{1}z}+p_{2}{\rm e}^{\alpha_{2}z}
\end{equation*}
and
\begin{equation*}
f^{n}(z)f'(z)+q(z){\rm e}^{Q(z)}f(z+c)=p_{1}{\rm e}^{\alpha_{1}z}+p_{2}{\rm e}^{\alpha_{2}z},
\end{equation*}
where $n\geq 3$ is an integer, $q(z)$ and $Q(z)$ are nonzero polynomials with $Q(z)$ nonconstant,
and $c,~p_{j},~\alpha_{j}(j=1,2)$ are nonzero constants satisfying $\alpha_{1}\neq \alpha_{2}$.
We sharpen the result in the cited literature, as illustrated by specific examples.\par

\par

{\bf 2020 MR Subject Classification}\quad 39A45; 30D30
\end{minipage}

\section{Introduction and Main results}

\indent\par Nevanlinna's value distribution theory has become a standard tool for studying the existence of entire
and meromorphic solutions of complex differential, difference, and differential-difference equations
\cite{hkl, lay, laz, li, lz, y, yl}. Among the problems in this area, the existence of meromorphic solutions
of nonlinear differential equations is particularly difficult and has attracted considerable attention.
Yang and Li \cite{yl}, for instance, investigated the nonlinear differential equation
\begin{equation}\label{eq1.1}
f^{3}(z)+\frac{3}{4}f''(z)=-\frac{1}{4}\sin 3z,
\end{equation}
and proved the following result:\par
\noindent{\bf Theorem A} (\cite[Theorem~4]{yl})\textit{\quad
The difference equation \eqref{eq1.1} has exactly three entire solutions:
$f_{1}(z)=\sin z$, $f_{2}(z)=\frac{\sqrt{3}}{2}\cos z-\frac{1}{2}\sin z$,
$f_{3}(z)=-\frac{\sqrt{3}}{2}\cos z-\frac{1}{2}\sin z$.
}\par

The development of the difference analogue of logarithmic derivative lemma (see \cite{cf, hk})
has paved the way for applying Nevanlinna theory to difference equations and differential-difference
equations. Yang and Laine \cite{y} investigated the existence of solutions to certain nonlinear
difference equations and obtained the following result:\par
\noindent{\bf Theorem B} (\cite[Theorem~2.5]{y})\textit{\quad
A nonlinear difference equation
\begin{equation}\label{eq1.2}
f^{3}(z)+q(z)f(z+1)=c\sin bz,
\end{equation}
where $q(z)$ is a nonconstant polynomial and $b,~c$ are nonzero constants,
does not admit entire solutions of finite order. If $q(z)=q$ is a nonzero
constant, then Eq. \eqref{eq1.1} possesses three district entire solutions of
finite order, provided $b=3\pi n$ and $q^{3}=(-1)^{n+1}\frac{27}{4}c^{2}$
for a nonzero integer $n$.
}\par

Note that $\sin bz$ is a linear combination of ${\rm e}^{ibz}$ and
${\rm e}^{-ibz}$. The more general differential (or difference) equation
\begin{equation}\label{eq1.3}
f^{n}(z)+P_{d}(z,f)=p_{1}{\rm e}^{\alpha_{1}z}+p_{2}{\rm e}^{\alpha_{2}z},
\end{equation}
where $n\geq 2$ is a positive integer, $P_{d}(z,f)$ is a differential (or difference) polynomial in $f$
of degree $d\leq n-1$, $\alpha_{j}~(j=1,2)$ are nonzero constants, has attracted much attention
\cite{cm,lyz,lm,lu}.\par

The dominant term $f^{n}(z)$ in Eq.\eqref{eq1.3} is replaced by $f^{n}(z)f'(z)$. For further related results,
we refer the reader to \cite{cm2,z}. Motivated by this, we study the meromorphic solution of nonlinear
differential-difference equation
\begin{equation}\label{eq1.4}
f^{n}(z)f'(z)+q(z)f(z+c)=p_{1}{\rm e}^{\alpha_{1}z}+p_{2}{\rm e}^{\alpha_{2}z},
\end{equation}
and obtain the following result.\par
\begin{theorem}\label{th1.1}
Let $n\geq 3$ be a positive integer, let $q(z)$ be a nonzero polynomial, and
let $c,~p_{j},~\alpha_{j}~(j=1,2)$ be nonzero constants such that $\alpha_{1}\neq \alpha_{2}$ and
$|\frac{\alpha_{1}}{\alpha_{2}}+\frac{\alpha_{2}}{\alpha_{1}}|\leq n+1+\frac{1}{n+1}$.
If Eq. \eqref{eq1.4} admits a finite-order meromorphic solution $f(z)$,
then $f(z)=\gamma{\rm e}^{\beta z}$, where $\beta$ and $\gamma$ are nonzero constants satisfying
$(n+1)\beta=\alpha_{1}(\text{or}~\alpha_{2})$, $\gamma^{n+1}\beta=p_{1}(\text{or}~p_{2})$, and
$\frac{\alpha_{1}}{\alpha_{2}}=n+1(\text{or}~\frac{1}{n+1})$.
\end{theorem}

\begin{remark}\label{re1.1}
We conjecture that Theorem \ref{th1.1} remains valid without the assumption
$|\frac{\alpha_{1}}{\alpha_{2}}+\frac{\alpha_{2}}{\alpha_{1}}|\leq n+1+\frac{1}{n+1}$,
although we do not yet have a satisfactory approach to its proof.
\end{remark}

Wen et al. \cite{wh} adopted a distinct approach by classifying the finite order entire solutions
$f$ of the following equation
\begin{equation}\label{eq1.5}
f^{n}(z)+q(z){\rm e}^{Q(z)}f(z+c)=p(z),
\end{equation}
where $q,~Q,~p$ are polynomials, $n\geq 2$ is an integer and $c\in \mathbb{C}\setminus\{0\}$.
What happens if the coefficient of $f(z+c)$ in Eq.\eqref{eq1.4} is perturbed by an exponential factor?
In 2022, Wang et al.\cite{w} investigated this problem and determined the form of the
meromorphic solutions of the following nonlinear differential-difference equation,
\begin{equation}\label{eq1.6}
f^{n}(z)f'(z)+q(z){\rm e}^{Q(z)}f(z+c)=p_{1}{\rm e}^{\alpha_{1}z}+p_{2}{\rm e}^{\alpha_{2}z},
\end{equation}
and established the following result:\par
\noindent{\bf Theorem C} (\cite[Theorem~6]{w}) \textit{\quad
Let $n\geq 4$ be an integer, $q(z)$ be a rational function and $Q(z)$ be a nonconstant
polynomial. Suppose that $c,~p_{j},~\alpha_{j}(j=1,2)$ are nonzero constants with
$\alpha_{1}\neq \alpha_{2}$. If the Eq.\eqref{eq1.6} admits a finite order meromorphic
solution $f(z)$ with $\lambda(z)<\sigma(f)$ and $\lambda(1/f)<\sigma(f)$, then $f(z)$
satisfies $\sigma(f)=\deg Q(z)=1$. Furthermore, $f(z)=A_{1}{\rm e}^{\frac{\alpha_{2}}{n+1}z}$,
$Q(z)=(\alpha_{1}-\frac{\alpha_{2}}{n+1})z+a_{0}$ or $f(z)=A_{2}{\rm e}^{\frac{\alpha_{1}}{n+1}z}$,
$Q(z)=(\alpha_{2}-\frac{\alpha_{1}}{n+1})z+a_{0}$, where $A_{1},~A_{2},~a_{0}$ are constants.
}\par

It is therefore natural to ask the following question: can one obtain a complete characterization
of all meromorphic solutions of Eq.\eqref{eq1.6} by dropping the hypotheses $\lambda(f)<\sigma(f)$
and $\lambda(1/f)<\sigma(f)$ imposed in Theorem C and relaxing the constraint from $n\geq 4$ to $n\geq 3$.
In this paper, we address this question and obtain the following results.\par
\begin{theorem}\label{th1.2}
Let $n\geq 4$ be an integer, let $q(z)$ and $Q(z)$ be nonzero polynomials such that
$Q(z)$ is nonconstant, and let $c,~p_{j},~\alpha_{j}~(j=1,2)$ be nonzero constants with
$\alpha_{1}\neq \alpha_{2}$. If Eq. \eqref{eq1.6} admits a finite-order meromorphic
solution $f(z)$, then $f(z)$ is entire, $\sigma(f)=\deg Q(z)=1$, and $q(z)$
reduces to a constant $q$ such that
\begin{equation*}
f(z)=\gamma_{1}{\rm e}^{\frac{\alpha_{1}}{n+1}z},~Q(z)=\frac{(n+1)\alpha_{2}-\alpha_{1}}{n+1}z+Q(0),
~(n+1)p_{1}q^{n+1}{\rm e}^{(n+1)Q(0)+\alpha_{1}c}=\alpha_{1}p_{2}^{n+1},
\end{equation*}
or
\begin{equation*}
f(z)=\gamma_{2}{\rm e}^{\frac{\alpha_{2}}{n+1}z},~Q(z)=\frac{(n+1)\alpha_{1}-\alpha_{2}}{n+1}z+Q(0),
~(n+1)p_{2}q^{n+1}{\rm e}^{(n+1)Q(0)+\alpha_{2}c}=\alpha_{2}p_{1}^{n+1},
\end{equation*}
where $\gamma_{1},~\gamma_{2}$ are constants satisfying $\gamma_{j}^{n+1}\frac{\alpha_{j}}{n+1}=p_{j}~(j=1,2)$.
\end{theorem}

Example~\ref{ex1.1} below shows that differential-difference equations \eqref{eq1.6} satisfying the conditions of
Theorem~\ref{th1.2} do exist.\par
\begin{example}\label{ex1.1}
The equation $f^{4}(z)f'(z)+2{\rm e}^{3z}f(z+\ln 2)={\rm e}^{5z}+4{\rm e}^{4z}$
admits a finite-order meromorphic solution $f(z)={\rm e}^{z}$.
\end{example}

We now consider the case $n=3$ of Eq.\eqref{eq1.6},
\begin{equation}\label{eq1.7}
f^{3}(z)f'(z)+q(z){\rm e}^{Q(z)}f(z+c)=p_{1}{\rm e}^{\alpha_{1}z}+p_{2}{\rm e}^{\alpha_{2}z},
\end{equation}
and obtain the following results.\par
\begin{theorem}\label{th1.3}
Let $q(z)$ and $Q(z)$ be nonzero polynomials such that $Q(z)$ is nonconstant,
and let $c,~p_{j},~\alpha_{j}~(j=1,2)$ be nonzero constants with $\alpha_{1}\neq \alpha_{2}$.
If Eq.\eqref{eq1.7}
admits a finite-order meromorphic solution $f(z)$, then $f(z)$ is entire,
$\sigma(f)=\deg Q(z)=1$, $q(z)$ reduces to a constant $q$ and one of the following holds:\par
$(1)$ $Q'(z)=\frac{4\alpha_{2}-\alpha_{1}}{4}$, $4q^{4}p_{1}{\rm e}^{4Q(0)+\alpha_{1}c}=\alpha_{1}p_{2}^{4}$,
and $f(z)=\beta_{1}{\rm e}^{\frac{\alpha_{1}}{4}z}$ or
$Q'(z)=\frac{4\alpha_{1}-\alpha_{2}}{4}$, $4q^{4}p_{2}{\rm e}^{4Q(0)+\alpha_{2}c}=\alpha_{2}p_{1}^{4}$,
and $f(z)=\beta_{2}{\rm e}^{\frac{\alpha_{2}}{4}z}$,
where $q,~\beta_{1},~\beta_{2}$ are constants satisfying $\alpha_{j}\beta_{j}^{4}=4p_{j}~(j=1,2)$.

$(2)$ $Q'(z)=\frac{9\alpha_{1}}{4}$, $\alpha_{1}+\alpha_{2}=0$, ${\rm e}^{\alpha_{1}c}=\frac{3}{8}$,
$-2\alpha_{1}p_{1}^{6}=243q^{4}{\rm e}^{4Q(0)}p_{2}^{3}$ and
$f(z)=\beta_{3}{\rm e}^{\frac{3\alpha_{1}}{4}z}+\beta_{4}{\rm e}^{-\frac{\alpha_{1}}{4}z}$ or
$Q'(z)=-\frac{9\alpha_{1}}{4}$, $\alpha_{1}+\alpha_{2}=0$, ${\rm e}^{\alpha_{1}c}=\frac{8}{3}$,
$2\alpha_{1}p_{2}^{6}=243q^{4}{\rm e}^{4Q(0)}p_{1}^{3}$ and
$f(z)=\beta_{3}{\rm e}^{\frac{\alpha_{1}}{4}z}+\beta_{4}{\rm e}^{-\frac{3\alpha_{1}}{4}z}$,
where $q,~\beta_{3},~\beta_{4}$ are constants satisfying
$-\frac{\alpha_{1}}{4}\beta_{3}^{4}=p_{2},~\frac{3\alpha_{1}}{2}\beta_{3}^{2}\beta_{4}^{2}=p_{1}$
or $\frac{\alpha_{1}}{4}\beta_{3}^{4}=p_{1}$, $-\frac{3\alpha_{1}}{2}\beta_{3}^{2}\beta_{4}^{2}=p_{2}$.

$(3)$ $Q'(z)=-\frac{\alpha_{1}+\alpha_{2}}{2}$, ${\rm e}^{\alpha_{1}c}=2$,
$2^{11}p_{2}^{9}q^{4}{\rm e}^{4Q(0)}=\alpha_{1}p_{1}^{12}$ and
$f(z)=\beta_{5}{\rm e}^{-\frac{\alpha_{1}}{2}z}+\beta_{6}{\rm e}^{\frac{\alpha_{1}}{2}z}$,
$\alpha_{2}=2\alpha_{1}$ or ${\rm e}^{\alpha_{2}c}=2$,
$2^{11}p_{1}^{9}q^{4}{\rm e}^{4Q(0)}=\alpha_{2}p_{2}^{12}$ and
$f(z)=\beta_{5}{\rm e}^{-\frac{\alpha_{2}}{2}z}+\beta_{6}{\rm e}^{\frac{\alpha_{2}}{2}z}$,
$\alpha_{1}=2\alpha_{2}$,
where $q,~\beta_{5},~\beta_{6}$ are constants satisfying $\alpha_{1}\beta_{5}\beta_{6}^{3}=p_{1},~
\alpha_{1}\beta_{6}^{4}=2p_{2}$ or $\alpha_{2}\beta_{5}\beta_{6}^{3}=p_{2},~
\alpha_{2}\beta_{6}^{4}=2p_{1}$.

$(4)$ $Q'(z)=\alpha_{1}+\alpha_{2}$, ${\rm e}^{\alpha_{1}c}=\frac{1}{3}$,
$3^{5}p_{1}^{6}q^{4}{\rm e}^{4Q(0)}=\alpha_{1}p_{2}^{9}$ and
$f(z)=\beta_{7}+\beta_{8}{\rm e}^{\alpha_{1}z}$, $\alpha_{2}=2\alpha_{1}$ or
${\rm e}^{\alpha_{2}c}=\frac{1}{3}$,
$3^{5}p_{2}^{6}q^{4}{\rm e}^{4Q(0)}=\alpha_{2}p_{1}^{9}$ and
$f(z)=\beta_{7}+\beta_{8}{\rm e}^{\alpha_{2}z}$, $\alpha_{1}=2\alpha_{2}$,
where $q,~\beta_{7},~\beta_{8}$ are constants satisfying $\alpha_{1}\beta_{7}^{3}\beta_{8}=p_{1},~
3\alpha_{1}\beta_{7}^{2}\beta_{8}^{2}=p_{2}$ or $\alpha_{2}\beta_{7}^{3}\beta_{8}=p_{2},~
3\alpha_{2}\beta_{7}^{2}\beta_{8}^{2}=p_{1}$.

$(5)$ $Q'(z)=\frac{\alpha_{1}+\alpha_{2}}{7}$, ${\rm e}^{\alpha_{1}c}=27$,
$3^{9}p_{2}^{5}q^{4}{\rm e}^{4Q(0)}=\alpha_{1}p_{1}^{8}$ and
$f(z)=\beta_{9}+\beta_{10}{\rm e}^{\frac{\alpha_{1}}{3}z}$, $\alpha_{2}=\frac{4}{3}\alpha_{1}$ or
${\rm e}^{\alpha_{2}c}=27$,
$3^{9}p_{1}^{5}q^{4}{\rm e}^{4Q(0)}=\alpha_{2}p_{2}^{8}$ and
$f(z)=\beta_{9}+\beta_{10}{\rm e}^{\frac{\alpha_{2}}{3}z}$, $\alpha_{1}=\frac{4}{3}\alpha_{2}$,
where $q,~\beta_{9},~\beta_{10}$ are constants satisfying $\alpha_{1}\beta_{9}\beta_{10}^{3}=p_{1},~
\alpha_{1}\beta_{10}^{4}=3p_{2}$ or $\alpha_{2}\beta_{9}\beta_{10}^{3}=p_{2},~
\alpha_{2}\beta_{10}^{4}=3p_{1}$.

$(6)$ $Q'(z)=\frac{2(\alpha_{1}+\alpha_{2})}{5}$, ${\rm e}^{\alpha_{1}c}=1$,
$3^{12}\alpha_{1}^{3}p_{1}^{4}p_{2}^{5}=q^{12}{\rm e}^{12Q(0)}$ and
$f(z)=\beta_{11}+\beta_{12}{\rm e}^{\alpha_{1}z}$, $\alpha_{2}=4\alpha_{1}$ or
${\rm e}^{\alpha_{2}c}=1$,
$3^{12}\alpha_{2}^{3}p_{2}^{4}p_{1}^{5}=q^{12}{\rm e}^{12Q(0)}$ and
$f(z)=\beta_{11}+\beta_{12}{\rm e}^{\alpha_{2}z}$, $\alpha_{1}=4\alpha_{2}$,
where $q,~\beta_{11},~\beta_{12}$ are constants satisfying $\alpha_{1}\beta_{11}^{3}\beta_{12}=p_{1},~
\alpha_{1}\beta_{12}^{4}=p_{2}$ or $\alpha_{2}\beta_{11}^{3}\beta_{12}=p_{2},~
\alpha_{2}\beta_{12}^{4}=p_{1}$.

$(7)$ $Q'(z)=\frac{5(\alpha_{1}+\alpha_{2})}{8}$, ${\rm e}^{\alpha_{1}c}=\frac{3}{4}$,
$48\alpha_{1}p_{1}p_{2}^{2}=q^{4}{\rm e}^{4Q(0)}$ and
$f(z)=\beta_{13}{\rm e}^{\frac{\alpha_{1}}{4}z}+\beta_{14}{\rm e}^{-\frac{3\alpha_{1}}{4}z}$
$\alpha_{2}=-3\alpha_{1}$,
or ${\rm e}^{\alpha_{2}c}=\frac{3}{4}$,
$48\alpha_{2}p_{2}p_{1}^{2}=q^{4}{\rm e}^{4Q(0)}$ and
$f(z)=\beta_{13}{\rm e}^{\frac{\alpha_{2}}{4}z}+\beta_{14}{\rm e}^{-\frac{3\alpha_{2}}{4}z}$
$\alpha_{1}=-3\alpha_{2}$,
where $q,~\beta_{13},~\beta_{14}$ are constants satisfying $\alpha_{1}\beta_{13}^{4}=4p_{1},
-3\alpha_{1}\beta_{14}^{4}=4p_{2}$ or $\alpha_{2}\beta_{13}^{4}=4p_{2},
-3\alpha_{2}\beta_{14}^{4}=4p_{1}$.
\end{theorem}

Examples \ref{ex1.2}--\ref{ex1.8} below show that there exist nonlinear differential-difference
equations of the form \eqref{eq1.7} satisfying the conditions of Theorem \ref{th1.3}.\par

\begin{example}\label{ex1.2}
The function $f(z)={\rm e}^{z}$ solves the differential-difference equation
\begin{equation*}
f^{3}(z)f'(z)+2i{\rm e}^{z}f(z+i\frac{\pi}{2})={\rm e}^{4z}-2{\rm e}^{2z},
\end{equation*}
which falls under conclusion $(1)$ of Theorem \ref{th1.3}.
\end{example}

\begin{example}\label{ex1.3}
The function $f(z)={\rm e}^{-z}+{\rm e}^{3z}$ solves the differential-difference equation
\begin{equation*}
f^{3}(z)f'(z)-{\rm e}^{9z+\frac{\ln 3+9\ln 2}{4}}
f\left(z+\frac{\ln 3-3\ln 2}{4}\right)=-{\rm e}^{-4z}+6{\rm e}^{4z},
\end{equation*}
which falls under conclusion $(2)$ of Theorem \ref{th1.3}.
\end{example}

\begin{example}\label{ex1.4}
The function $f(z)={\rm e}^{z}+{\rm e}^{-z}$ solves the differential-difference equation
\begin{equation*}
f^{3}(z)f'(z)-{\rm e}^{3z+\frac{1}{2}\ln2}
f\left(z-\frac{1}{2}\ln2\right)=-{\rm e}^{-4z}-2{\rm e}^{-2z},
\end{equation*}
which falls under conclusion $(3)$ of Theorem \ref{th1.3}.
\end{example}

\begin{example}\label{ex1.5}
The function $f(z)=1+{\rm e}^{z}$  solves the differential-difference equation
\begin{equation*}
f^{3}(z)f'(z)-{\rm e}^{3z+\ln3}f(z-\ln3)=3{\rm e}^{2z}+{\rm e}^{z},
\end{equation*}
which falls under conclusion $(4)$ of Theorem \ref{th1.3}.
\end{example}

\begin{example}\label{ex1.6}
The function $f(z)=1+{\rm e}^{z}$  solves the differential-difference equation
\begin{equation*}
f^{3}(z)f'(z)+{\rm e}^{z+i\pi}f(z+\ln3)={\rm e}^{4z}+3{\rm e}^{3z},
\end{equation*}
which falls under conclusion $(5)$ of Theorem \ref{th1.3}.
\end{example}

\begin{example}\label{ex1.7}
The function $f(z)=1+{\rm e}^{z}$  solves the differential-difference equation
\begin{equation*}
f^{3}(z)f'(z)+3{\rm e}^{2z+i\pi}f(z+2i\pi)={\rm e}^{4z}+{\rm e}^{z},
\end{equation*}
which falls under conclusion $(6)$ of Theorem \ref{th1.3}.
\end{example}

\begin{example}\label{ex1.8}
The function $f(z)={\rm e}^{-z}+{\rm e}^{3z}$  solves the differential-difference equation
\begin{equation*}
f^{3}(z)f'(z)-6\sqrt[4]{\frac{4}{3}}{\rm e}^{5z}
f\left(z+\frac{1}{4}\ln\frac{4}{3}\right)=-{\rm e}^{-4z}+3{\rm e}^{12z},
\end{equation*}
which falls under conclusion $(7)$ of Theorem \ref{th1.3}.
\end{example}

In this paper, we apply Cartan's second main theorem and Nevanlinna's theorem for families of
meromorphic functions to investigate the nonlinear differential-difference
equations \eqref{eq1.6} and \eqref{eq1.7}. For the standard notation and basic results of Nevanlinna theory, we refer the reader to \cite{ch,h,l,yy}.
Throughout, $\sigma(f)$, $\lambda(f)$ and $\lambda\Big(\frac{1}{f}\Big)$ denote, respectively,
the order of growth, the exponent of convergence of zeros and
the exponent of convergence  of poles of a meromorphic function $f(z)$.
We write $S(r,f)$ for any quantity satisfying $S(r,f)=o(T(r,f))$,
as $r\to\infty$, possibly outside of a set of $r$ of finite logarithmic measure.

\section{Some Lemmas}
\begin{lemma}[{See \cite[Lemma~2.4]{lm}}]\label{le2.1}
Let $n\geq 2$ be integer, let $\alpha_{j}~(j=1,2)$
be distinct nonzero constants, and let $p_{j}~(j=1,2)$ be nonzero meromorphic functions.
Then the equation
\begin{equation*}
f^{n}(z)=p_{1}{\rm e}^{\alpha_{1}z}+p_{2}{\rm e}^{\alpha_{2}z}
\end{equation*}
cannot admit a meromorphic solution $f$ such that $T(r,p_{j})=S(r,f)~(j=1,2)$.
\end{lemma}

\begin{lemma}[{See \cite[Lemma~2.6]{cm2}}]\label{le2.2}
Let $n\geq 2$ be integer, let $\alpha_{j}~(j=1,2)$
be distinct nonzero constants, and let $p_{j}~(j=1,2)$ be nonzero rational functions.
Then the equation
\begin{equation*}
f^{n}(z)f'(z)=p_{1}{\rm e}^{\alpha_{1}z}+p_{2}{\rm e}^{\alpha_{2}z}
\end{equation*}
has no meromorphic solutions.
\end{lemma}

\begin{lemma}[{See \cite[Theorems~2.1~and~2.2]{cf}}]\label{le2.3}
Let $f(z)$ be a meromorphic function of finite order, and let $\eta$ be a be a
nonzero complex number. Then
\begin{equation*}
\begin{split}
T(r,f(z+\eta))=T(r,f)+S(r,f),\\
N(r,f(z+\eta))=N(r,f)+S(r,f).
\end{split}
\end{equation*}
\end{lemma}

\begin{remark}\label{re2.1}
Under the conditions of Lemma \ref{le2.3}, by {\rm\cite[Lemma~$8.3$]{hk2}},
the following result holds:
\begin{equation*}
N\Big(r,\frac{1}{f(z+\eta)}\Big)=N\Big(r,\frac{1}{f}\Big)+S(r,f).
\end{equation*}
\end{remark}

\begin{lemma}[{See \cite[Corollary~2.6]{cf}}]\label{le2.4}
Let $\eta_{1},~\eta_{2}$ be two complex numbers such that $\eta_{1}\neq \eta_{2}$, and
let $f$ be a finite order meromorphic function. Let $\sigma$ be the order of $f$. Then for
each $\varepsilon>0$, we have
\begin{equation*}
m\Big(r,\frac{f(z+\eta_{1})}{f(z+\eta_{2})}\Big)=O(r^{\sigma-1+\varepsilon})+O(\log r).
\end{equation*}
\end{lemma}

\begin{lemma}[{See \cite[Lemma~2.2]{cm2}}]\label{le2.5}
Let $f$ be a transcendental meromorphic solution of the differential-difference equation
\begin{equation*}
f^n(z)f'(z)+P_d(z,f)=H(z),
\end{equation*}
where $n(\geq 2)$ is an integer, $P_d(z,f)$ is a difference polynomial in $f$ of degree
$d\leq n-1$, and $H(z)$ is meromorphic function satisfying $N(r,H)=S(r,f)$. If
$\sigma(f)<+\infty$, then $N(r,f)=S(r,f)$, and $\sigma(f)=\sigma(H)$.
\end{lemma}

\begin{lemma}[{See \cite[Corollary~3.3]{hk}}]\label{le2.6}
Let $f$ be a nonconstant finite order meromorphic solution of
$f^{n}(z)P(z,f)=Q(z,f)$,
where $P(z,f)$ and $Q(z,f)$ are difference polynomials
in $f$ with small meromorphic coefficients.
If the total degree of $Q(z,f)$ as a polynomial
in $f$ and its shifts is $\leq n$, then
\begin{equation*}
m(r,P(z,f))=S(r,f).
\end{equation*}
\end{lemma}

\begin{remark}\label{re2.2}
Lemma \ref{le2.6} is still holds for the case $P(z,f)$, $Q(z,f)$ being differential-difference
polynomials in $f$ with functions of small proximity related to $f$ as its coefficients.
\end{remark}

\begin{lemma}\label{le2.7}
Let $n\geq 3$ be an integer, let $q(z)$ and $Q(z)$ be nonzero polynomilas such that $Q(z)$ is nonconstant,
and let $c,~p_{j},~\alpha_{j}~(j=1,2)$ be nonzero constants with $\alpha_{1}\neq\alpha_{2}$.
If Eq.\eqref{eq1.6} admits a finite-order meromorphic solution $f(z)$,
then $f(z)$ is entire and $\sigma(f)=\deg Q(z)=1$.
\end{lemma}
\begin{proof}
Suppose that $f(z)$ has at least one pole. Let $z_{1}$ be a pole of $f(z)$ of multiplicity
$k_{1}(\geq 1)$. It follows from \eqref{eq1.6} that $z_{1}+c$ is a pole of $f(z)$ of
multiplicity $k_{2}>(n+1)k_{1}$. Iterating the above argument, we obtain a sequence $\{z_{1}+(j-1)c\}$
of poles of $f(z)$ with multiplicity $k_{j}>(n+1)^{j-1}k_{1}$ for $j=2,3,\ldots$, which implies
$\lambda(\frac{1}{f})=+\infty$. This contradicts $\sigma(f)<+\infty$.
Hence $f(z)$ is entire.\par

Suppose that $\deg Q(z)>1$, it follows from \eqref{eq1.6} and Lemma \ref{le2.3} that
\begin{equation}\label{eq2.1}
m(r,{\rm e}^{Q})\leq O(r)+(n+2)T(r,f)+S(r,f).
\end{equation}
Then $f(z)$ is transcendental entire and
\begin{equation}\label{eq2.2}
S(r,{\rm e}^{Q})\subset S(r,f).
\end{equation}
Denote $G(z)=p_{1}{\rm e}^{\alpha_{1}z}+p_{2}{\rm e}^{\alpha_{2}z}$. Differentiating \eqref{eq1.6}
and eliminating ${\rm e}^{Q(z)}$ yields
\begin{equation}\label{eq2.3}
f^{n-1}\psi=qG\Big(\frac{G'}{G}-\frac{q'}{q}-\frac{f'(z+c)}{f(z+c)}-Q'\Big)f(z+c),
\end{equation}
where
\begin{equation}\label{eq2.4}
\psi=nqf(z+c)(f')^{2}+qff''f(z+c)-[q'f(z+c)+qQ'f(z+c)+qf'(z+c)]ff'.
\end{equation}
We deduce that $\psi\not\equiv 0$. Indeed, if $\psi\equiv 0$, then by \eqref{eq2.3},
we would have $dq(z){\rm e}^{Q(z)}f(z+c)=G(z)$, where $d\in\mathbb{C}\setminus\{0\}$ is a constant.
Substituting this into \eqref{eq1.6} yields $f^{n}(z)f'(z)=(1-\frac{1}{d})G(z)$.
If $d=1$, then $f^{n}(z)f'(z)=0$, which is impossible.
If $d\neq 1$, we can also get a contradiction by Lemma \ref{le2.2}.
Noting that $\deg Q(z)>1$, by the lemma on the logarithmic derivative and Lemmas \ref{le2.3}
and \ref{le2.4}, we have
\begin{equation}\label{eq2.5}
m\Big(r,qG\Big(\frac{G'}{G}-\frac{q'}{q}-\frac{f'(z+c)}{f(z+c)}-Q'\Big)\Big)=S(r,f).
\end{equation}
Then, by \eqref{eq2.3}--\eqref{eq2.5}, Lemma \ref{le2.6} and Remark \ref{re2.2}, we have
\begin{equation}\label{eq2.6}
m(r,\psi)=S(r,f),\quad m(r,f\psi)=S(r,f).
\end{equation}
Since $f(z)$ is entire, then $N(r,\psi)=O(\log r)$. It follows from \eqref{eq2.6} that
\begin{equation*}
\begin{split}
T(r,f)=m(r,f)& \leq m(r,f\psi)+m\Big(r,\frac{1}{\psi}\Big)\\
& \leq m(r,\psi)+N(r,\psi)+S(r,f)=S(r,f),
\end{split}
\end{equation*}
which is impossible. So $\deg Q(z)=1$. Write $Q(z)=\alpha z+\beta$, where $\alpha$ and $\beta$
are constants with $\alpha\neq 0$.\par

Finally, we show that $\sigma(f)=1$. If $\sigma(f)<1$, it follows from \eqref{eq1.6} that
\begin{equation}\label{eq2.7}
p_{1}{\rm e}^{\alpha_{1}z}+p_{2}{\rm e}^{\alpha_{2}z}-q(z)f(z+c){\rm e}^{\alpha z+\beta}-f^{n}(z)f'(z)=0.
\end{equation}
If $\alpha=\alpha_{j}$, $j\in\{1,2\}$ or $\alpha\neq\alpha_{j}$, $j=1,2$,
by Lemma \ref{le2.8}, we get $f^{n}(z)f'(z)\equiv 0$, a contradiction. If $\sigma(f)>1$,
\eqref{eq1.6} can be rewritten  as
\begin{equation}\label{eq2.8}
f^{n}(z)f'(z)=p_{1}{\rm e}^{\alpha_{1}z}+p_{2}{\rm e}^{\alpha_{2}z}-q(z){\rm e}^{\alpha z+\beta}f(z+c).
\end{equation}
Since $T(r,q(z){\rm e}^{\alpha z+\beta})=S(r,f)$ and $T(r,p_{1}{\rm e}^{\alpha_{1}z}+p_{2}{\rm e}^{\alpha_{2}z})=S(r,f)$,
it follows from Lemma \ref{le2.6} and Remark \ref{re2.2} that $m(r,f')=S(r,f)$ and $m(r,ff')=S(r,f)$.
Then $T(r,f)=m(r,f)\leq m(r,ff')+m(r,\frac{1}{f'})\leq m(r,f')+N(r,f')+S(r,f)=S(r,f)$, which is impossible.
Hence $\sigma(f)=1$.\par
\end{proof}

\begin{lemma}[{See \cite[Theorem~1.51]{yy}}]\label{le2.8}
Suppose that $f_{1}, f_{2},\ldots,f_{n}(n\geq 2)$ are meromorphic functions and
$g_{1}, g_{2},\ldots,g_{n}$ are entire functions satisfying the following conditions:\\
$(i)$ $\sum_{j=1}^{n}\limits f_{j}{\rm e}^{g_{j}}\equiv 0.$\\
$(ii)$ $g_{j}-g_{k}$ are not constants for $1\leq j<k\leq n$.\\
$(iii)$ For $1\leq j\leq n, 1\leq h<k\leq n$,
$$T(r,f_{j})=o(T(r,{\rm e}^{g_{h}-g_{k}}))~(r\rightarrow \infty, r\not\in E),$$
where $E$ is a set with finite linear measure.\\
Then $f_{j}\equiv 0~(j=1,\ldots, n)$.
\end{lemma}

Lemmas \ref{le2.9} and \ref{le2.10} below are essential to our estimates of the
integrated counting functions of the solutions of \eqref{eq1.6} and \eqref{eq1.7}.
Throughout, $E$ denotes a set of finite linear measure, not
necessarily the same at each occurrence. For a nonconstant meromorphic function
$f$ and a positive integer $p$, we write $n_{p}\Big(r,\tfrac{1}{f}\Big)$ for
the number of zeros of $f$ in $\{|z|\leq r\}$, where a zero of multiplicity $q$
is counted $\min\{q,p\}$ times, and $N_{p}\Big(r,\tfrac{1}{f}\Big)$ for the
corresponding integrated counting function.

\begin{lemma}[{See \cite[Theorem~2.1]{g},\cite{c}}]\label{le2.9}
Let $f_{1},f_{2},\ldots,f_{p}$ be linearly independent entire functions, where $p\geq 2$.
Suppose that for each complex number $z$, $\max\{|f_{1}(z)|,\ldots,|f_{p}(z)|\}>0$.
For positive $r$, set
\begin{equation*}
T(r)=\frac{1}{2\pi}\int_{0}^{2\pi}u(re^{i\theta})d\theta-u(0),\quad
\text{where}\quad u(z)=\sup\limits_{1\leq j\leq p}\log|f_{j}(z)|.
\end{equation*}
Set $f_{p+1}=f_{1}+\cdots+f_{p}$. Then
\begin{equation*}
T(r)\leq \sum_{j=1}^{p+1}N_{p-1}\Big(r,\frac{1}{f_{j}}\Big)+S(r)\leq
(p-1)\sum_{j=1}^{p+1}\overline{N}\Big(r,\frac{1}{f_{j}}\Big)+S(r),
\end{equation*}
where $S(r)$ is a quantity satisfying $S(r)=O(\log(T(r)))+O(\log r)$
($r\to \infty,~r\not\in E$). If at least one of the quotients $f_{j}/f_{m}$
is a transcendental function, then $S(r)=o(T(r))$, ($r\to \infty,~r\not\in E$),
while if all the quotients $f_{j}/f_{m}$ are rational functions, then
$S(r)=-\frac{1}{2}p(p-1)\log r+O(1)$, ($r\to \infty,~r\not\in E$).
\end{lemma}

\begin{lemma}[{See \cite[Lemma~2.2]{g},\cite{c}}]\label{le2.10}
Assume that the hypotheses of Lemma \ref{le2.9} hold. Then for any $j$ and $m$,
\begin{equation*}
T(r,f_{j}/f_{m})\leq T(r)+O(1) \quad (r\to \infty),
\end{equation*}
and for any $j$,
\begin{equation*}
N(r,1/f_{j})\leq T(r)+O(1) \quad (r\to \infty).
\end{equation*}
\end{lemma}

\section{ Proof of Theorem \ref{th1.1}}
\begin{proof}
Suppose that $f$ is a finite-order meromorphic solution of \eqref{eq1.4}. We claim
that $f$ is an entire function. Otherwise, $f$ has a pole $z_1$ of
multiplicity $k_1(\geq 1)$. By \eqref{eq1.4} and the assumptions of Theorem \ref{th1.1},
$z_2:=z_1+c$ is a pole of $f(z)$ of multiplicity $k_2>(n+1)k_1$.
Iterating this argument, we obtain a sequence of poles $\{z_{j}=z_{1}+(j-1)c\}$ of $f(z)$ with multiplicity
$k_j>(n+1)^{j-1}k_{1}$ for $j=2,3,\ldots$, which implies $\lambda(\frac{1}{f})=+\infty$.
This contradicts $\sigma(f)<+\infty$. Hence $f$ is entire.
Since $p_{1}{\rm e}^{\alpha_{1}z}+p_{2}{\rm e}^{\alpha_{2}z}$
is transcendental, $f$ cannot be a polynomial and is therefore transcendental.
Note that $N(r,p_{1}{\rm e}^{\alpha_{1}z}+p_{2}{\rm e}^{\alpha_{2}z})=0$ and
$\sigma(p_{1}{\rm e}^{\alpha_{1}z}+p_{2}{\rm e}^{\alpha_{2}z})=1$, it then follows from Lemma \ref{le2.5}
that $\sigma(f)=1$.

Denoting $F(z):=f^{n}(z)f'(z)$, $P(z):=q(z)f(z+c)$ and differentiating \eqref{eq1.4}, we have
\begin{align}\label{eq3.1}
F'+P'=\alpha_{1}p_{1}{\rm e}^{\alpha_{1}z}+\alpha_{2}p_{2}{\rm e}^{\alpha_{2}z}.
\end{align}
Eliminating ${\rm e}^{\alpha_{1}z}$ from \eqref{eq1.4} and \eqref{eq3.1} yields
\begin{align}\label{eq3.2}
(\alpha_{1}F-F')+(\alpha_{1}P-P')=p_{2}(\alpha_{1}-\alpha_{2}){\rm e}^{\alpha_{2}z}.
\end{align}
Differentiating \eqref{eq3.2}, we have
\begin{align}\label{eq3.3}
(\alpha_{1}F'-F'')+(\alpha_{1}P'-P'')=p_{2}(\alpha_{1}-\alpha_{2})\alpha_{2}{\rm e}^{\alpha_{2}z}.
\end{align}
Eliminating ${\rm e}^{\alpha_{2}z}$ from \eqref{eq3.2} and \eqref{eq3.3} yields
\begin{align}\label{eq3.4}
\alpha_{1}\alpha_{2}F-(\alpha_{1}+\alpha_{2})F'+F''=
-[\alpha_{1}\alpha_{2}P-(\alpha_{1}+\alpha_{2})P'+P''].
\end{align}
We now consider two cases:\par
{\bf Case 1:}\ \ Suppose that $\alpha_{1}\alpha_{2}F-(\alpha_{1}+\alpha_{2})F'+F''\equiv 0$. By
solving a second-order differential equation, we get that
\begin{align}\label{eq3.5}
f^{n}(z)f'(z)=d_{1}{\rm e}^{\alpha_{1}z}+d_{2}{\rm e}^{\alpha_{2}z},
\end{align}
where $d_{1}$ and $d_{1}$ are constants. By Lemma \ref{le2.2}, we have
$f^{n}(z)f'(z)=d_{1}{\rm e}^{\alpha_{1}z}$ or $f^{n}(z)f'(z)=d_{2}{\rm e}^{\alpha_{2}z}$.
If $f^{n}(z)f'(z)=d_{1}{\rm e}^{\alpha_{1}z}$, then $f(z)=\gamma{\rm e}^{\beta z}$.
Substituting this expression into \eqref{eq1.4}, we get $\gamma^{n+1}\beta=p_{1}$,
$q(z)\gamma{\rm e}^{\beta c}=p_{2}$, $(n+1)\beta=\alpha_{1}$, $\beta=\alpha_{2}$,
which implies that $q(z)$ reduces to a constant $q$.
Similarly, if $f^{n}(z)f'(z)=d_{2}{\rm e}^{\alpha_{2}z}$, we obtain $\gamma^{n+1}\beta=p_{2}$,
$q\gamma{\rm e}^{\beta c}=p_{1}$, $(n+1)\beta=\alpha_{2}$, $\beta=\alpha_{1}$.

{\bf Case 2:}\ \
Suppose that $\alpha_{1}\alpha_{2}F-(\alpha_{1}+\alpha_{2})F'+F''\not\equiv 0$.
Substituting $F=f^{n}f'$, $F'=nf^{n-1}(f')^{2}+f^{n}f''$,
and $F''=n(n-1)f^{n-2}(f')^{3}+3nf^{n-1}f'f''+f^{n}f'''$ into
\eqref{eq1.4} yields
\begin{align}\label{eq3.6}
f^{n-2}\varphi=-[\alpha_{1}\alpha_{2}P-(\alpha_{1}+\alpha_{2})P'+P''],
\end{align}
where
\begin{align}\label{eq3.7}
\varphi=\alpha_{1}\alpha_{2}f^{2}f'-
(\alpha_{1}+\alpha_{2})[f(f')^{2}+f^{2}f'']+n(n-1)(f')^{3}+
3nff'f''+f^{2}f'''.
\end{align}
Note that $f$ is entire and $\sigma(f)=1$, by \eqref{eq3.6} and the same proof of
Lemma \ref{le2.6}, we have
\begin{align}\label{eq3.8}
T(r,\varphi)=m(r,\varphi)=O(\log r).
\end{align}
It follows from \eqref{eq3.7} that any multiple zero of $f$
must also be a zero of $\varphi$. Then
$N_{(2}(r,\frac{1}{f})\leq N(r,\frac{1}{\varphi})\leq T(r,\varphi)+O(1)=O(\log r)$.
Consequently,
\begin{align}\label{eq3.9}
T(r,f)=N_{1)}\Big(r,\frac{1}{f}\Big)+O(\log r).
\end{align}
Differentiating \eqref{eq3.7} results in
\begin{align}\label{eq3.10}
\begin{split}
& \varphi'=2\alpha_{1}\alpha_{2}f(f')^{2}+\alpha_{1}\alpha_{2}f^{2}f''
-(\alpha_{1}+\alpha_{2})[(f')^{3}+4ff'f''+f^{2}f''']+3n^{2}(f')^{2}f''\\
& +3nf(f'')^{2}+(3n+2)ff'f'''+f^{2}f^{(4)}.
\end{split}
\end{align}
Suppose that $z_0$ is a simple zero of $f$. By \eqref{eq3.7} and \eqref{eq3.10},
we see that $z_0$ is a zero of $[n(n-1)\varphi'+(\alpha_{1}+\alpha_{2})\varphi]f'-3n^{2}\varphi f''$.
Denote
\begin{align}\label{eq3.11}
\tau(z)=\frac{[n(n-1)\varphi'+(\alpha_{1}+\alpha_{2})\varphi]f'-3n^{2}\varphi f''}{f}.
\end{align}
Then $T(r,\tau)=O(\log r)$, hence $\tau(z)$ is a rational function.
It follows from \eqref{eq3.11} that
\begin{align}\label{eq3.12}
f''=A_1f'+B_1f,
\end{align}
where $A_1=\frac{(n-1)\varphi'}{3n\varphi}+\frac{\alpha_{1}+\alpha_{2}}{3n^{2}}$,
$B_1=-\frac{\tau}{3n^{2}\varphi}$ are rational functions.
It follows from \eqref{eq3.12} that
\begin{align}\label{eq3.13}
f'''=(A_1^{2}+A'_1+B_1)f'+(A_1B_1+B'_1)f,
\end{align}
\begin{align}\label{eq3.14}
f^{(4)}=(A_1^{3}+3A_1A'_1+2A_1B_1+A''_1+2B'_1)f'+(A_1^{2}B_1+2A'_1B_1+A_1B'_1+B_1^{2}+B''_1)f.
\end{align}
Substituting \eqref{eq3.12} and \eqref{eq3.13} into \eqref{eq3.7} yields
\begin{align}\label{eq3.15}
\varphi=s_1(f')^{3}+s_2f(f')^{2}+s_3f^{2}f'+s_4f^{3},
\end{align}
where
\begin{align*}
\left\{\begin{array}{l}
s_1=n(n-1),\\
s_2=3nA_1-(\alpha_{1}+\alpha_{2}),\\
s_3=A_1^{2}+A'_1+(3n+1)B_1-(\alpha_{1}+\alpha_{2})A_1+\alpha_{1}\alpha_{2},\\
s_4=A_1B_1+B'_1-(\alpha_{1}+\alpha_{2})B_1,
\end{array}\right.
\end{align*}
are rational functions.
Substituting \eqref{eq3.12}-\eqref{eq3.14} into \eqref{eq3.10} yields
\begin{align}\label{eq3.16}
\varphi'=t_1(f')^{3}+t_2f(f')^{2}+t_3f^{2}f'+t_4f^{3},
\end{align}
where
\begin{align*}
\left\{\begin{array}{l}
t_1=3n^{2}A_1-(\alpha_{1}+\alpha_{2}),\\
t_2=2(3n+1)A_1^{2}+(3n+2)A'_1+(3n^{2}+3n+2)B_1-4(\alpha_{1}+\alpha_{2})A_1+2\alpha_{1}\alpha_{2},\\
t_3=A_1^{3}+3A_1A'_1+(9n+4)A_1B_1+A''_1+\alpha_{1}\alpha_{2}A_1+(3n+4)B'_1\\
-(\alpha_{1}+\alpha_{2})(A_1^{2}+A'_1+5B_1),\\
t_4=A_1^{2}B_1+2A'_1B_1+A_1B'_1+(3n+1)B_1^{2}+\alpha_{1}\alpha_{2}B_1+B''_1-(\alpha_{1}+\alpha_{2})(A_1B_1+B'_1),
\end{array}\right.
\end{align*}
are rational functions.
Based on the expressions of $s_{1},~t_{1}$ and $A_{1}$, it follows that
$t_1-s_1\frac{\varphi'}{\varphi}=3n^{2}A_1-(\alpha_{1}+\alpha_{2})-n(n-1)\frac{\varphi'}{\varphi}\equiv 0$.
Combining \eqref{eq3.15} and \eqref{eq3.16}, we have
\begin{align}\label{eq5.19}
\Big(t_2-s_2\frac{\varphi'}{\varphi}\Big)f(f')^{2}+
\Big(t_3-s_3\frac{\varphi'}{\varphi}\Big)f^{2}f'+
\Big(t_4-s_4\frac{\varphi'}{\varphi}\Big)f^{3}=0.
\end{align}
It follows from \eqref{eq3.9} that
\begin{align*}
& t_2-s_2\frac{\varphi'}{\varphi}
=2(3n+1)A_1^{2}+(3n+2)A'_1+(3n^{2}+3n+2)B_1-4(\alpha_{1}+\alpha_{2})A_1+2\alpha_{1}\alpha_{2}\\
& -[3nA_1-(\alpha_{1}+\alpha_{2})]\frac{\varphi'}{\varphi}\equiv 0.\\
& t_3-s_3\frac{\varphi'}{\varphi}
=A_1^{3}+3A_1A'_1+(9n+4)A_1B_1+A''_1+\alpha_{1}\alpha_{2}A_1+(3n+4)B'_1\\
& -(\alpha_{1}+\alpha_{2})(A_1^{2}+A'_1+5B_1)-
[A_1^{2}+A'_1+(3n+1)B_1-(\alpha_{1}+\alpha_{2})A_1+\alpha_{1}\alpha_{2}]\frac{\varphi'}{\varphi}\equiv 0.\\
& t_4-s_4\frac{\varphi'}{\varphi}
=A_1^{2}B_1+2A'_1B_1+A_1B'_1+(3n+1)B_1^{2}+\alpha_{1}\alpha_{2}B_1+B''_1-(\alpha_{1}+\alpha_{2})(A_1B_1+B'_1)\\
& -[A_1B_1+B'_1-(\alpha_{1}+\alpha_{2})B_1]\frac{\varphi'}{\varphi}\equiv 0.
\end{align*}
It follows from $t_2-s_2\frac{\varphi'}{\varphi}\equiv 0$ that
\begin{align*}
& 2(3n+1)A_1^{2}+(3n+2)A'_1+(3n^{2}+3n+2)B_1-4(\alpha_{1}+\alpha_{2})A_1+2\alpha_{1}\alpha_{2}\\
& =[3nA_1-(\alpha_{1}+\alpha_{2})]\frac{\varphi'}{\varphi}.
\end{align*}
As $z\to \infty$, we have $\frac{\varphi'}{\varphi}\to 0$, and hence
$A_1=\frac{(n-1)\varphi'}{3n\varphi}+\frac{\alpha_{1}+\alpha_{2}}{3n^{2}}\to \frac{\alpha_{1}+\alpha_{2}}{3n^{2}}$,
$A'_1\to 0$. Recalling that $B_1=-\frac{\tau}{3n^{2}\varphi}$, we obtain from the above limits that
\begin{align}\label{eq3.18}
2(3n+1)\Big(\frac{\alpha_{1}+\alpha_{2}}{3n^{2}}\Big)^{2}+(3n^{2}+3n+2)\Big(-\frac{\tau}{3n^{2}\varphi}\Big)
-\frac{4(\alpha_{1}+\alpha_{2})^{2}}{3n^{2}}+2\alpha_{1}\alpha_{2}\to 0.
\end{align}
It follows from $t_4-s_4\frac{\varphi'}{\varphi}\equiv 0$ that
\begin{align*}
& A_1^{2}B_1+2A'_1B_1+A_1B'_1+(3n+1)B_1^{2}+\alpha_{1}\alpha_{2}B_1+B''_1-(\alpha_{1}+\alpha_{2})(A_1B_1+B'_1)\\
& =[A_1B_1+B'_1-(\alpha_{1}+\alpha_{2})B_1]\frac{\varphi'}{\varphi}.
\end{align*}
As $z\to \infty$, we get $\frac{\varphi'}{\varphi}\to 0$, and hence
$A_1\to \frac{\alpha_{1}+\alpha_{2}}{3n^{2}}$, $A'_1\to 0$,
$\frac{B_1'}{B_1}\to 0$, $\frac{B''_1}{B_1}\to 0$. Recalling that $B_1=-\frac{\tau}{3n^{2}\varphi}$,
we get from the above limits that
\begin{align}\label{eq3.19}
\Big(\frac{\alpha_{1}+\alpha_{2}}{3n^{2}}\Big)^{2}+(3n+1)\Big(-\frac{\tau}{3n^{2}\varphi}\Big)
-\frac{(\alpha_{1}+\alpha_{2})^{2}}{3n^{2}}+\alpha_{1}\alpha_{2}\to 0.
\end{align}
It follows from \eqref{eq3.18} and \eqref{eq3.19} that
\begin{align}\label{eq3.20}
\lim_{z\to \infty}\frac{\tau}{\varphi}=\frac{3n^{2}}{3n^{2}+3n+2}
\Big[2(3n+1)\Big(\frac{\alpha_{1}+\alpha_{2}}{3n^{2}}\Big)^{2}
-\frac{4(\alpha_{1}+\alpha_{2})^{2}}{3n^{2}}+2\alpha_{1}\alpha_{2}\Big],
\end{align}
and
\begin{align}\label{eq3.21}
\lim_{z\to \infty}\frac{\tau}{\varphi}=\frac{3n^{2}}{3n+1}
\Big[\Big(\frac{\alpha_{1}+\alpha_{2}}{3n^{2}}\Big)^{2}
-\frac{(\alpha_{1}+\alpha_{2})^{2}}{3n^{2}}+\alpha_{1}\alpha_{2}\Big].
\end{align}
Combining \eqref{eq3.20} and \eqref{eq3.21}, we obtain
$\frac{\alpha_{1}}{\alpha_{2}}+\frac{\alpha_{2}}{\alpha_{1}}=3n^{2}+6n+13+
\frac{3(12n+5)}{n^{2}-2n-1}>n+1+\frac{1}{n+1}$ $(n\geq 3)$,
which contradicts the hypothesis
$\Big|\frac{\alpha_{1}}{\alpha_{2}}+\frac{\alpha_{2}}{\alpha_{1}}\Big|\leq n+1+\frac{1}{n+1}$.
This completes the proof.
\par
\end{proof}

\section{ Proof of Theorem \ref{th1.2}}
\begin{proof}
By Lemma \ref{le2.7}, we see that $f(z)$ is a transcendental entire function, and $\sigma(f)=\deg Q(z)=1$.
We now distinguish two cases.\par
{\Case 1} Suppose that $p_{1}{\rm e}^{\alpha_{1}z}$, $p_{2}{\rm e}^{\alpha_{2}z}$, $-q(z){\rm e}^{Q(z)}f(z+c)$
are linearly dependent. Then there exist constants $d_{1}$ and $d_{2}$, not both zero, such that
\begin{equation}\label{eq4.1}
q(z){\rm e}^{Q(z)}f(z+c)=d_{1}p_{1}{\rm e}^{\alpha_{1}z}+d_{2}p_{2}{\rm e}^{\alpha_{2}z},
\end{equation}
By \eqref{eq1.6} and Lemma \ref{le2.2}, we get
\begin{equation}\label{eq4.2}
f^{n}(z)f'(z)=(1-d_{1})p_{1}{\rm e}^{\alpha_{1}z}\quad \text{or} \quad
f^{n}(z)f'(z)=(1-d_{2})p_{2}{\rm e}^{\alpha_{2}z},
\end{equation}
where $d_{1}(\neq 1)$ and $d_{2}(\neq 1)$ are constants. It follows from \eqref{eq4.2} that
\begin{equation}\label{eq4.3}
f(z)=\gamma{\rm e}^{\frac{\alpha}{n+1}z},
\end{equation}
where $(\alpha,\gamma^{n+1}\frac{\alpha}{n+1})\in\{(\alpha_1,(1-d_1)p_1),(\alpha_2,(1-d_2)p_2)\}$.
Substituting \eqref{eq4.3} into \eqref{eq1.6} yields
\begin{equation}\label{eq4.4}
\gamma^{n+1}\frac{\alpha}{n+1}{\rm e}^{\alpha z}
+\gamma{\rm e}^{\frac{\alpha c}{n+1}}q(z){\rm e}^{\frac{\alpha z+(n+1)Q(z)}{n+1}}
=p_{1}{\rm e}^{\alpha_{1}z}+p_{2}{\rm e}^{\alpha_{2}z}.
\end{equation}
By \eqref{eq4.4} and Lemma \ref{le2.8}, we get
\begin{equation*}
Q(z)=\frac{(n+1)\alpha_{2}-\alpha_{1}}{n+1}z+Q(0),~\gamma{\rm e}^{\frac{(n+1)Q(0)+\alpha_{1}c}{n+1}}q(z)=p_{2},
~\gamma^{n+1}\frac{\alpha_{1}}{n+1}=p_{1},
\end{equation*}
or
\begin{equation*}
Q(z)=\frac{(n+1)\alpha_{1}-\alpha_{2}}{n+1}z+Q(0),~\gamma{\rm e}^{\frac{(n+1)Q(0)+\alpha_{2}c}{n+1}}q(z)=p_{1},
~\gamma^{n+1}\frac{\alpha_{2}}{n+1}=p_{2}.
\end{equation*}
The above equalities imply that $q(z)$ reduces to a constant
and that $(n+1)p_{1}q^{n+1}{\rm e}^{(n+1)Q(0)+\alpha_{1}c}=\alpha_{1}p_{2}^{n+1}$,
respectively $(n+1)p_{2}q^{n+1}{\rm e}^{(n+1)Q(0)+\alpha_{2}c}=\alpha_{2}p_{1}^{n+1}$.
Hence the conclusion of Theorem \ref{th1.2} holds.\par

{\Case 2} Suppose that $p_{1}{\rm e}^{\alpha_{1}z}$, $p_{2}{\rm e}^{\alpha_{2}z}$, $-q(z){\rm e}^{Q(z)}f(z+c)$
are linearly independent. Since $f$ is entire, by \eqref{eq1.6}, Lemmas \ref{le2.9} and \ref{le2.10},
remark \ref{re2.1}, for sufficiently large $r\not\in E$, we obtain
\begin{equation}\label{eq4.5}
\begin{split}
nN\Big(r,\frac{1}{f}\Big) &=N\Big(r,\frac{1}{f^{n}}\Big)-N(r,f')
\leq N\Big(r,\frac{1}{f^{n}f'}\Big)\leq T_{1}(r)\\
& \leq N_{2}\Big(r,\frac{1}{p_{1}{\rm e}^{\alpha_{1}z}}\Big)+
N_{2}\Big(r,\frac{1}{p_{2}{\rm e}^{\alpha_{2}z}}\Big)+
N_{2}\Big(r,\frac{1}{q{\rm e}^{Q}f(z+c)}\Big)\\
& +N_{2}\Big(r,\frac{1}{f^{n}f'}\Big)+o( T_{1}(r))\\
& \leq N\Big(r,\frac{1}{f(z+c)}\Big)+N_{2}\Big(r,\frac{1}{f^{n}f'}\Big)+o( T_{1}(r))+O(\log r)\\
& \leq 3N\Big(r,\frac{1}{f}\Big)+o( T_{1}(r))+S(r,f),
\end{split}
\end{equation}
where
\begin{equation*}
\begin{split}
T_{1}(r)=\frac{1}{2\pi}\int_{0}^{2\pi}u_{1}(r{\rm e}^{i\theta})d\theta-u_{1}(0),\quad
u_{1}(z)=\sup\limits_{1\leq j\leq 2}\{\log|p_{j}{\rm e}^{\alpha_{j}z}|,\log|q(z){\rm e}^{Q(z)}f(z+c)|\}.
\end{split}
\end{equation*}
Since $n\geq 4$, it follows from \eqref{eq4.5} that
\begin{equation}\label{eq4.6}
N\left(r,\frac{1}{f}\right)=S(r,f).
\end{equation}
Differentiating \eqref{eq1.6} and eliminating ${\rm e}^{\alpha_{1}z}$ yields
\begin{equation}\label{eq4.7}
nf^{n-1}(f')^{2}+f^{n}f''-\alpha_{1}f^{n}f'+\zeta(z)q(z){\rm e}^{Q(z)}f(z+c)
=(\alpha_{2}-\alpha_{1}){\rm e}^{\alpha_{2}z},
\end{equation}
where $\zeta(z)=\frac{q'(z)}{q(z)}+\frac{f'(z+c)}{f(z+c)}+Q'(z)-\alpha_{1}\not\equiv 0$.
Otherwise, $\zeta(z)\equiv 0$, then $q(z){\rm e}^{Q(z)}f(z+c)=d_{3}{\rm e}^{\alpha_{1}z}$,
$d_{3}\in\mathbb{C}\setminus\{0\}$, which contradicts the fact that
$p_{1}{\rm e}^{\alpha_{1}z}$, $p_{2}{\rm e}^{\alpha_{2}z}$,
$-q(z){\rm e}^{Q(z)}f(z+c)$ are linearly independent.
Differentiating \eqref{eq4.7} and eliminating ${\rm e}^{\alpha_{2}z}$ yields
\begin{equation}\label{eq4.8}
\begin{split}
& n(n-1)f^{n-2}(f')^{3}+3nf^{n-1}f'f''+f^{n}f'''-n(\alpha_{1}+\alpha_{2})f^{n-1}(f')^{2}
-(\alpha_{1}+\alpha_{2})f^{n}f''\\
& +\alpha_{1}\alpha_{2}f^{n}f'
=\Big(\alpha_{2}-\alpha_{1}-\frac{\zeta'}{\zeta}-\zeta\Big)\zeta q{\rm e}^{Q(z)}f(z+c).
\end{split}
\end{equation}

If $\alpha_{2}-\alpha_{1}-\frac{\zeta'}{\zeta}-\zeta\equiv 0$, then by \eqref{eq4.8}, we get
\begin{equation*}
n(n-1)f^{n-2}(f')^{3}+3nf^{n-1}f'f''+f^{n}f'''-n(\alpha_{1}+\alpha_{2})f^{n-1}(f')^{2}
-(\alpha_{1}+\alpha_{2})f^{n}f''+\alpha_{1}\alpha_{2}f^{n}f'\equiv 0,
\end{equation*}
that is,
\begin{equation}\label{eq4.9}
(f^{n}f')''-(\alpha_{1}+\alpha_{2})(f^{n}f')'+\alpha_{1}\alpha_{2}f^{n}f'\equiv 0.
\end{equation}
Solving \eqref{eq4.9}, one obtains
\begin{equation}\label{eq4.10}
f^{n}f'=d_{4}{\rm e}^{\alpha_{1}z}+d_{5}{\rm e}^{\alpha_{2}z},
\end{equation}
where $d_{4},~d_{5}$ are constants. By \eqref{eq4.10} and Lemma \ref{le2.2}, we get
$d_{4}\neq 0,~d_{5}=0$ or $d_{4}=0,~d_{5}\neq 0$. Then $f(z)$ has the form \eqref{eq4.3}.
By an argument similar to that used in Case 1, we obtain the conclusion of Theorem \ref{th1.2}.\par

If $\alpha_{2}-\alpha_{1}-\frac{\zeta'}{\zeta}-\zeta\not\equiv 0$,
then by \eqref{eq4.6}, Lemma \ref{le2.3}, remark \ref{re2.1} and the expression of $\zeta(z)$, we have
\begin{equation}\label{eq4.11}
T(r,\zeta)=N(r,\zeta)+S(r,f)\leq N\Big(r,\frac{1}{f(z+c)}\Big)+S(r,f)
\leq N\Big(r,\frac{1}{f}\Big)+S(r,f)=S(r,f).
\end{equation}
Combining \eqref{eq4.6}, \eqref{eq4.8} and \eqref{eq4.11} with Lemma \ref{le2.4}, we obtain
\begin{equation}\label{eq4.12}
\begin{split}
m\Big(r,\frac{{\rm e}^{Q(z)}}{f^{n}}\Big) & \leq m\Big(r,n(n-1)\Big(\frac{f'}{f}\Big)^{3}+3n\frac{f'}{f}\cdot\frac{f''}{f}
+\frac{f'''}{f}-n(\alpha_{1}+\alpha_{2})\Big(\frac{f'}{f}\Big)^{2}\\
& -(\alpha_{1}+\alpha_{2})\frac{f''}{f}+\alpha_{1}\alpha_{2}\frac{f'}{f}\Big)+m\Big(r,\frac{f(z)}{f(z+c)}\Big)+S(r,f)\\
& =S(r,f).
\end{split}
\end{equation}
Since $f$ is entire, combining \eqref{eq1.6}, \eqref{eq4.6} and \eqref{eq4.12}, we obtain
\begin{equation}\label{eq4.13}
\begin{split}
T\Big(r,\frac{p_{1}{\rm e}^{\alpha_{1}z}+p_{2}{\rm e}^{\alpha_{2}z}}{f^{n+1}}\Big) & =
T\Big(r,\frac{f'}{f}+\frac{q(z){\rm e}^{Q(z)}f(z+c)}{f^{n+1}}\Big)\\
& =m\Big(r,\frac{f'}{f}+\frac{q(z){\rm e}^{Q(z)}}{f^{n}}\cdot\frac{f(z+c)}{f(z)}\Big)
+N\Big(r,\frac{f'}{f}+\frac{q(z){\rm e}^{Q(z)}f(z+c)}{f^{n+1}}\Big)\\
& =S(r,f).
\end{split}
\end{equation}
Denote $\psi(z)=\frac{f^{n+1}(z)}{p_{1}{\rm e}^{\alpha_{1}z}+p_{2}{\rm e}^{\alpha_{2}z}}$, then
we have $f^{n+1}(z)=p_{1}\psi(z){\rm e}^{\alpha_{1}z}+p_{2}\psi(z){\rm e}^{\alpha_{2}z}$.
It follows from \eqref{eq4.13} that $T(r,p_{j}\psi(z))=S(r,f)~(j=1,2)$,
which contradicts Lemma \ref{le2.1}. This completes the proof.\par

\end{proof}

\section{ Proof of Theorem \ref{th1.3}}
\begin{proof}
By Lemma \ref{le2.7}, we see that $f(z)$ is a transcendental entire function, and $\sigma(f)=\deg Q(z)=1$.
Set $Q(z)=az+b$, where $a(\neq 0),~b$ are constants. Then, Eq.\eqref{eq1.7} can be rewritten as
\begin{equation}\label{eq5.1}
f^{3}(z)f'(z)+q(z){\rm e}^{az+b}f(z+c)=p_{1}{\rm e}^{\alpha_{1}z}+p_{2}{\rm e}^{\alpha_{2}z},
\end{equation}
In what follows, we consider two cases.\par
{\Case 1} Suppose that $p_{1}{\rm e}^{\alpha_{1}z}$, $p_{2}{\rm e}^{\alpha_{2}z}$, $-q{\rm e}^{az}f(z+c)$
are linearly dependent. Arguing as in Case 1 of the proof of Theorem \ref{th1.2},
we obtain from \eqref{eq5.1} and Lemma \ref{le2.2},
\begin{equation}\label{eq5.2}
f^{3}(z)f'(z)=(1-d_{1})p_{1}{\rm e}^{\alpha_{1}z},~d_{1}\neq 1,
\end{equation}
or
\begin{equation}\label{eq5.3}
f^{3}(z)f'(z)=(1-d_{2})p_{2}{\rm e}^{\alpha_{2}z},~d_{2}\neq 1,
\end{equation}
where $d_{1}$ and $d_{2}$ are constants. If \eqref{eq5.2} holds, then
$f(z)=\eta_{1}{\rm e}^{\frac{\alpha_{1}}{4}z}$, where $\eta_{1}^{4}\frac{\alpha_{1}}{4}=(1-d_{1})p_{1}$.
Substituting this into \eqref{eq5.1} yields
\begin{equation*}
\frac{\alpha_{1}}{4}\eta_{1}^{4}{\rm e}^{\alpha_{1}z}
+\eta_{1}{\rm e}^{b+\frac{\alpha_{1}}{4}c}q(z){\rm e}^{\frac{4a+\alpha_{1}}{4}z}
=p_{1}{\rm e}^{\alpha_{1}z}+p_{2}{\rm e}^{\alpha_{2}z},
\end{equation*}
then $\frac{\alpha_{1}}{4}\eta_{1}^{4}=p_{1}$, $\eta_{1}{\rm e}^{b+\frac{\alpha_{1}}{4}c}q(z)=p_{2}$
and $a=\frac{4\alpha_{2}-\alpha_{1}}{3}$. Furthermore, we find that $q(z)$ reduces to a constant and that
$4p_{1}q^{4}{\rm e}^{4b+\alpha_{1}c}=\alpha_{1}p_{2}^{4}$. Similarly, if \eqref{eq5.3} holds,
we obtain
\begin{equation*}
f(z)=\eta_{2}{\rm e}^{\frac{\alpha_{2}}{4}z} \quad \text{and}\quad
\frac{\alpha_{2}}{4}\eta_{2}^{4}=p_{2},~4p_{2}q^{4}{\rm e}^{4b+\alpha_{2}c}=\alpha_{2}p_{1}^{4},
~a=\frac{4\alpha_{1}-\alpha_{2}}{4}.
\end{equation*}
Then part $(1)$ of Theorem \ref{th1.3} holds.\par

{\Case 2} Suppose that $p_{1}{\rm e}^{\alpha_{1}z}$, $p_{2}{\rm e}^{\alpha_{2}z}$, $-q(z){\rm e}^{az+b}f(z+c)$
are linearly independent. Since $f$ is entire, by \eqref{eq5.1}, Lemmas \ref{le2.9} and \ref{le2.10},
remark \ref{re2.1}, for sufficiently large $r\not\in E$, we get
\begin{equation}\label{eq5.4}
\begin{split}
3N\Big(r,\frac{1}{f}\Big) & =N\Big(r,\frac{1}{f^{3}}\Big)-N(r,f')\leq N\Big(r,\frac{1}{f^{3}f'}\Big)\leq T_{1}(r)\\
& \leq N_{2}\Big(r,\frac{1}{p_{1}{\rm e}^{\alpha_{1}z}}\Big)+
N_{2}\Big(r,\frac{1}{p_{2}{\rm e}^{\alpha_{2}z}}\Big)+
N_{2}\Big(r,\frac{1}{q(z){\rm e}^{az+b}f(z+c)}\Big)\\
& +N_{2}\Big(r,\frac{1}{f^{3}f'}\Big)+o( T_{2}(r))\\
& \leq N\Big(r,\frac{1}{f(z+c)}\Big)+N_{2}\Big(r,\frac{1}{f^{3}f'}\Big)+o( T_{2}(r))+O(\log r)\\
& = N\Big(r,\frac{1}{f}\Big)+N_{2}\Big(r,\frac{1}{f^{3}f'}\Big)+o( T_{2}(r))+S(r,f),
\end{split}
\end{equation}
where
\begin{equation*}
\begin{split}
T_{2}(r)=\frac{1}{2\pi}\int_{0}^{2\pi}u_{2}(r{\rm e}^{i\theta})d\theta-u_{2}(0),\quad
u_{2}(z)=\sup\limits_{1\leq j\leq 2}\{\log|p_{j}{\rm e}^{\alpha_{j}z}|,\log|q(z){\rm e}^{az+b}f(z+c)|\}.
\end{split}
\end{equation*}
It follows from \eqref{eq5.4} that
\begin{equation}\label{eq5.5}
T_{2}(r)=O(T(r,f))~(r\to\infty,~r\not\in E).
\end{equation}
Let $z_{0}$ be a zero of $f(z)$ of multiplicity $l\geq 2$. Then the multiplicity of $z_{0}$
of $2n(r,\frac{1}{f})-n_{2}(r,\frac{1}{f^{3}f'})$ is $2l-2(\geq l)$. It follows from
\eqref{eq5.4} and \eqref{eq5.5} that
\begin{equation}\label{eq5.6}
N_{(2}\left(r,\frac{1}{f}\right)=S(r,f),
\end{equation}
where $N_{(2}(r,\frac{1}{f})$ denotes the counting function of multiple zeros of $f(z)$.
Differentiating \eqref{eq5.1} and eliminating ${\rm e}^{\alpha_{1}z}$ yields
\begin{equation}\label{eq5.7}
\begin{split}
& 3f^{2}(f')^{2}+f^{3}f''-\alpha_{1}f^{3}f'+{\rm e}^{az+b}\{(a-\alpha_{1})qf(z+c)+qf'(z+c)+q'f(z+c)\}\\
& =(\alpha_{2}-\alpha_{1})p_{2}{\rm e}^{\alpha_{2}z}.
\end{split}
\end{equation}
Differentiating \eqref{eq5.7} and eliminating ${\rm e}^{\alpha_{2}z}$ yields
\begin{equation}\label{eq5.8}
\begin{split}
& f[6(f')^{3}+9ff'f''+\alpha_{1}\alpha_{2}f^{2}f'-(\alpha_{1}+\alpha_{2})f^{2}f''+f^{2}f'''
-3(\alpha_{1}+\alpha_{2})f(f')^{2}]\\
& =-{\rm e}^{az+b}[B_{0}f(z+c)+B_{1}f'(z+c)+B_{2}f''(z+c)],
\end{split}
\end{equation}
where $B_{0}=(a-\alpha_{1})(a-\alpha_{2})q(z)+(2a-\alpha_{1}-\alpha_{2})q'(z)+q''(z)$,
$B_{1}=(2a-\alpha_{1}-\alpha_{2})q(z)+2q'(z)$, $B_{2}=q(z)$.
Next we consider two subcases.\par
{\bf Subcase 2.1} If $6(f')^{3}+9ff'f''+f^{2}f'''
-(\alpha_{1}+\alpha_{2})f^{2}f''-3(\alpha_{1}+\alpha_{2})f(f')^{2}+\alpha_{1}\alpha_{2}f^{2}f'\equiv 0$,
then we get $(f^{3}f')''-(\alpha_{1}+\alpha_{2})(f^{3}f')'+\alpha_{1}\alpha_{2}f^{3}f'=0$,
which implies that $f^{3}f'=d_{3}{\rm e}^{\alpha_{1}z}+d_{4}{\rm e}^{\alpha_{2}z}$.
By Lemma \ref{le2.2}, we get
\begin{equation}\label{eq5.9}
f^{3}f'=d_{3}{\rm e}^{\alpha_{1}z}\quad \text{or} \quad f^{3}f'=d_{4}{\rm e}^{\alpha_{2}z},
\end{equation}
where $d_{3},~d_{4}$ are constants. Substituting \eqref{eq5.9} into \eqref{eq5.1} yields
\begin{equation*}
-q(z){\rm e}^{az+b}f(z+c)=(d_{3}-p_{1}){\rm e}^{\alpha_{1}z}-p_{2}{\rm e}^{\alpha_{2}z} \quad
\text{or} \quad
-q(z){\rm e}^{az+b}f(z+c)=(d_{4}-p_{2}){\rm e}^{\alpha_{2}z}-p_{1}{\rm e}^{\alpha_{1}z},
\end{equation*}
which contradicts the linear independence of $p_{1}{\rm e}^{\alpha_{1}z}$, $p_{2}{\rm e}^{\alpha_{2}z}$
and $-q(z){\rm e}^{az+b}f(z+c)$.\par

{\bf Subcase 2.2}
If $6(f')^{3}+9ff'f''+f^{2}f'''
-(\alpha_{1}+\alpha_{2})f^{2}f''-3(\alpha_{1}+\alpha_{2})f(f')^{2}+\alpha_{1}\alpha_{2}f^{2}f'\not\equiv 0$,
it follows from \eqref{eq5.8} that every zero of $f(z)$ must be a zero of
$B_{0}f(z+c)+B_{1}f'(z+c)+B_{2}f''(z+c)$. Set
\begin{equation}\label{eq5.10}
\tau_{1}=\frac{B_{0}f(z+c)+B_{1}f'(z+c)+B_{2}f''(z+c)}{f(z)}.
\end{equation}
By \eqref{eq5.6} and \eqref{eq5.10}, the lemma on the logarithmic derivatives
and Lemma \ref{le2.4}, we have
\begin{equation}\label{eq5.11}
T(r,\tau_{1})=S(r,f).
\end{equation}
It follows from \eqref{eq5.8} and \eqref{eq5.10} that
\begin{equation}\label{eq5.12}
6(f')^{3}+9ff'f''+f^{2}f'''
-(\alpha_{1}+\alpha_{2})f^{2}f''-3(\alpha_{1}+\alpha_{2})f(f')^{2}+\alpha_{1}\alpha_{2}f^{2}f'
=-\tau_{1}{\rm e}^{az+b}.
\end{equation}
By \eqref{eq5.12}, $\tau_{1}$ is an entire function with $\tau_{1}\not\equiv 0$. Differentiating \eqref{eq5.12}
and eliminating ${\rm e}^{az+b}$ yields
\begin{equation}\label{eq5.13}
\begin{split}
& f\{[2\alpha_{1}\alpha_{2}+3A(\alpha_{1}+\alpha_{2})](f')^{2}+9(f'')^{2}-[8(\alpha_{1}+\alpha_{2})+9A]f'f''
+11f'f'''-A\alpha_{1}\alpha_{2}ff'+\\
& [\alpha_{1}\alpha_{2}+A(\alpha_{1}+\alpha_{2})]ff''-(A+\alpha_{1}+\alpha_{2})ff'''+ff^{(4)}\}
=3[(2A+\alpha_{1}+\alpha_{2})f'-9f''](f')^{2},
\end{split}
\end{equation}
where $A=a+\frac{\tau'_{1}}{\tau_{1}}$. We claim that $(2A+\alpha_{1}+\alpha_{2})f'-9f''\not\equiv 0$.
Otherwise, $(2A+\alpha_{1}+\alpha_{2})f'-9f''\equiv 0$, then we have
\begin{equation}\label{eq5.14}
f'(z)=d_{5}(\tau_{1})^{\frac{2}{9}}{\rm e}^{\frac{2a+\alpha_{1}+\alpha_{2}}{9}z},
\end{equation}
where $d_{5}$ is a nonzero constant. It follows from \eqref{eq5.13} that
\begin{equation}\label{eq5.15}
\begin{split}
& [2\alpha_{1}\alpha_{2}+3A(\alpha_{1}+\alpha_{2})](f')^{2}+9(f'')^{2}-[8(\alpha_{1}+\alpha_{2})+9A]f'f''+11f'f'''\\
& =f\{A\alpha_{1}\alpha_{2}f'-[\alpha_{1}\alpha_{2}+A(\alpha_{1}+\alpha_{2})]f''+(A+\alpha_{1}+\alpha_{2})f'''-f^{(4)}\}.
\end{split}
\end{equation}
Combining \eqref{eq5.14} and \eqref{eq5.15}, we obtain
\begin{equation}\label{eq5.16}
f(z)=\eta_{3}(z){\rm e}^{\frac{2a+\alpha_{1}+\alpha_{2}}{9}z},
\end{equation}
where $\eta_{3}(z)$ is an entire function with $T(r,\eta_{3})=S(r,f)$. Substituting \eqref{eq5.16} into \eqref{eq5.1} yields
\begin{equation}\label{eq5.17}
\begin{split}
& \eta^{3}_{3}(z)\Big[\eta'_{3}(z)+\frac{2a+\alpha_{1}+\alpha_{2}}{9}\eta_{3}(z)\Big]{\rm e}^{\frac{4(2a+\alpha_{1}+\alpha_{2})}{9}z}+\\
&  {\rm e}^{b+\frac{2a+\alpha_{1}+\alpha_{2}}{9}c}q(z)\eta_{3}(z+c){\rm e}^{\frac{11a+\alpha_{1}+\alpha_{2}}{9}z}
=p_{1}{\rm e}^{\alpha_{1}z}+p_{2}{\rm e}^{\alpha_{2}z}.
\end{split}
\end{equation}
By \eqref{eq5.17} and Lemma \ref{le2.8}, we get
${\rm e}^{b+\frac{2a+\alpha_{1}+\alpha_{2}}{9}c}q(z)\eta_{3}(z+c){\rm e}^{\frac{11a+\alpha_{1}+\alpha_{2}}{9}z}
=q(z){\rm e}^{az+b}f(z+c)=p_{1}{\rm e}^{\alpha_{1}z}$ (or $p_{2}{\rm e}^{\alpha_{2}z}$), which contradicts the linear
independence of $p_{1}{\rm e}^{\alpha_{1}z}$, $p_{2}{\rm e}^{\alpha_{2}z}$ and $-q(z){\rm e}^{az+b}f(z+c)$.
Then $(2A+\alpha_{1}+\alpha_{2})f'-9f''\not\equiv 0$. Let
\begin{equation}\label{eq5.18}
\zeta=\frac{(2A+\alpha_{1}+\alpha_{2})f'-9f''}{f}.
\end{equation}
It follows from \eqref{eq5.13} that every simple zero of $f(z)$ must be a zero of
$(2A+\alpha_{1}+\alpha_{2})f'-9f''$. By \eqref{eq5.6}, \eqref{eq5.11}, \eqref{eq5.18} and the lemma
on the logarithmic derivatives, we get
\begin{equation}\label{eq5.19}
T(r,\zeta)=S(r,f),\quad \zeta\not\equiv 0.
\end{equation}
Substituting \eqref{eq5.18} into \eqref{eq5.13} yields
\begin{equation}\label{eq5.20}
\begin{split}
& [2\alpha_{1}\alpha_{2}+3A(\alpha_{1}+\alpha_{2})-3\zeta](f')^{2}+9(f'')^{2}
-[8(\alpha_{1}+\alpha_{2})+9A]f'f''+11f'f'''\\
& -A\alpha_{1}\alpha_{2}ff'+[\alpha_{1}\alpha_{2}+A(\alpha_{1}+\alpha_{2})]ff''
-(A+\alpha_{1}+\alpha_{2})ff'''+ff^{(4)}=0.
\end{split}
\end{equation}
It follows from \eqref{eq5.18} that
\begin{equation}\label{eq5.21}
f''(z)=u_{0}(z)f(z)+u_{1}(z)f'(z),
\end{equation}
\begin{equation}\label{eq5.22}
f'''(z)=u_{2}(z)f(z)+u_{3}(z)f'(z),
\end{equation}
\begin{equation}\label{eq5.23}
f^{(4)}(z)=u_{4}(z)f(z)+u_{5}(z)f'(z),
\end{equation}
where $u_{0}(z)=-\frac{\zeta(z)}{9}$, $u_{1}(z)=\frac{2A+\alpha_{1}+\alpha_{2}}{9}$,
$u_{2}(z)=u_{1}(z)u_{0}(z)+u'_{0}(z)$, $u_{3}(z)=u_{1}^{2}(z)+u'_{1}(z)+u_{0}(z)$,
$u_{4}(z)=u_{1}^{2}(z)u_{0}(z)+2u'_{1}(z)u_{0}(z)+u_{1}(z)u'_{0}(z)+u_{0}^{2}(z)+u''_{0}(z)$,
$u_{5}(z)=u_{1}^{3}(z)+3u_{1}(z)u'_{1}(z)+2u_{1}(z)u_{0}(z)+u''_{1}(z)+2u'_{0}(z)$,
$T(r,u_{j})=S(r,f),~j=0,1,2,\ldots,5$.
Substituting \eqref{eq5.21}--\eqref{eq5.23}  into \eqref{eq5.20}, we get
\begin{equation}\label{eq5.24}
D_{3}(f')^{2}+D_{2}ff'+D_{1}f^{2}=0,
\end{equation}
where $D_{1}(z)=9u_{0}^{2}(z)+[\alpha_{1}\alpha_{2}+A(\alpha_{1}+\alpha_{2})]u_{0}(z)-
(A+\alpha_{1}+\alpha_{2})u_{2}(z)+u_{4}(z)$,
$D_{2}(z)=18u_{1}(z)u_{0}(z)-[8(\alpha_{1}+\alpha_{2})+9A]u_{0}(z)+11u_{2}(z)-A\alpha_{1}\alpha_{2}
+[\alpha_{1}\alpha_{2}+A(\alpha_{1}+\alpha_{2})]u_{1}(z)-
(A+\alpha_{1}+\alpha_{2})u_{3}(z)+u_{5}(z)$,
$D_{3}(z)=2\alpha_{1}\alpha_{2}+3A(\alpha_{1}+\alpha_{2})-3\zeta(z)+9u_{1}^{2}(z)-
[8(\alpha_{1}+\alpha_{2})+9A]u_{1}(z)+11u_{3}(z)$, and $T(r,D_{j})=S(r,f)(j=1,2,3)$.
We claim that
\begin{equation*}
D_{3}(z)\equiv 0,\quad D_{2}(z)\equiv 0 \quad\text{and}\quad D_{1}(z)\equiv 0.
\end{equation*}
Otherwise, $D_{3}(z)\not\equiv 0$. It follows from \eqref{eq5.24} that
\begin{equation*}
2N\Big(r,\frac{f'}{f}\Big)=N\Big(r,\Big(\frac{f'}{f}\Big)^{2}\Big)
=N\Big(r,\frac{D_{2}}{D_{3}}\cdot\frac{f'}{f}+\frac{D_{1}}{D_{3}}\Big)
\leq N\Big(r,\frac{D_{2}}{D_{3}}\Big)+N\Big(r,\frac{D_{1}}{D_{3}}\Big)+N\Big(r,\frac{f'}{f}\Big),
\end{equation*}
hence,
\begin{equation*}
\overline{N}(r,f)+\overline{N}\Big(r,\frac{1}{f}\Big)=N\Big(r,\frac{f'}{f}\Big)\leq S(r,f).
\end{equation*}
Combining this with \eqref{eq5.6}, we obtain
\begin{equation*}
N\Big(r,\frac{1}{f}\Big)\leq \overline{N}\Big(r,\frac{1}{f}\Big)+
N_{(2}\Big(r,\frac{1}{f}\Big)=S(r,f).
\end{equation*}
From the above inequality, we obtain
\begin{equation}\label{eq5.25}
N\Big(r,\frac{p_{1}{\rm e}^{\alpha_{1}z}+p_{2}{\rm e}^{\alpha_{2}z}}{f^{4}}\Big)=S(r,f).
\end{equation}
It follows from \eqref{eq5.11} and \eqref{eq5.12} that
\begin{equation*}
m\Big(r,\frac{{\rm e}^{az+b}}{f^{3}}\Big)=S(r,f).
\end{equation*}
Then by \eqref{eq5.1}, one can obtain
\begin{equation}\label{eq5.26}
m\Big(r,\frac{p_{1}{\rm e}^{\alpha_{1}z}+p_{2}{\rm e}^{\alpha_{2}z}}{f^{4}}\Big)=S(r,f).
\end{equation}
It follows from \eqref{eq5.25} and \eqref{eq5.26} that
\begin{equation}\label{eq5.27}
T\Big(r,\frac{p_{1}{\rm e}^{\alpha_{1}z}+p_{2}{\rm e}^{\alpha_{2}z}}{f^{4}}\Big)=S(r,f).
\end{equation}
Denote $\omega(z):=\frac{f^{4}(z)}{p_{1}{\rm e}^{\alpha_{1}z}+p_{2}{\rm e}^{\alpha_{2}z}}$,
we obtain $f^{4}(z)=p_{1}\omega(z){\rm e}^{\alpha_{1}z}+p_{2}\omega(z){\rm e}^{\alpha_{2}z}$,
$T(r,p_{j}\omega)=S(r,f)(j=1,2)$. This contradicts Lemma \ref{le2.1}. So $D_{3}(z)\equiv 0$.
By an analogous argument, $D_{2}(z)\equiv 0$.
It then follows from \eqref{eq5.24} that $D_{1}(z)\equiv 0$. Set
\begin{equation}\label{eq5.28}
\lambda:=\alpha_{1}+\alpha_{2}.
\end{equation}
Substituting \eqref{eq5.28} and the expressions for $u_{j}(z)(j=0,1,2,\ldots,5)$ into
$D_{3}(z)\equiv 0$ and $D_{2}(z)\equiv 0$, we obtain the following, respectively:
\begin{equation}\label{eq5.29}
u_{1}(z)=\frac{2A+\lambda}{9},\quad
u_{0}(z)=\frac{1}{38}\Big\{\frac{82A^{2}-98\lambda A+52\lambda^{2}-198A'}{81}
-2\alpha_{1}\alpha_{2}\Big\},
\end{equation}
\begin{equation}\label{eq5.30}
560A^{3}-420\lambda A^{2}-(840\lambda^{2}-2835\alpha_{1}\alpha_{2})A-3780AA'+945\lambda A'
+2835A''+484\lambda^{3}-1863\lambda\alpha_{1}\alpha_{2}\equiv 0.
\end{equation}
Note that $A=a+\frac{\tau'_{1}}{\tau_{1}}$, where $\tau_{1}$ is an entire function with $\tau_{1}\not\equiv 0$.
Suppose that $z_{0}$ is a zero of $\tau_{1}(z)$ of multiplicity $m\geq 1$.
Then in the punctured neighborhood $U^{o}(z_{0})$ of $z_{0}$, we get
\begin{equation}\label{eq5.31}
A(z)=\frac{m}{z-z_{0}}+\tau_{2}(z),
\end{equation}
where $\tau_{2}(z)$ is analytic in $U^{o}(z_{0})$. Substituting \eqref{eq5.31} into
\eqref{eq5.30} and comparing the coefficients of $\frac{m}{(z-z_{0})^{3}}$ gives that
$8m^{3}+54m^{2}+81m=0$. Then $m=-\frac{9}{2}$ or $m=-\frac{9}{4}$. This contradicts
the fact that $m$ is an integer. Thus, $\tau_{1}(z)$ is an
entire function without zeros. Since $\sigma(f)=1$, it follows from \eqref{eq5.11}
that $\tau_{1}(z)$ is a constant. From \eqref{eq5.30} and \eqref{eq5.21},
we obtain, respectively,
\begin{equation}\label{eq5.32}
81(23\lambda-35a)\alpha_{1}\alpha_{2}=4(140a^{3}-105\lambda a^{2}-210\lambda^{2}a+121\lambda^{3}),
\end{equation}
\begin{equation}\label{eq5.33}
f''(z)-\frac{2a+\lambda}{9}f'(z)-
\frac{1}{38}\Big(\frac{82a^{2}-98\lambda a+52\lambda^{2}}{81}-2\alpha_{1}\alpha_{2}\Big)f(z)=0.
\end{equation}
It follows from \eqref{eq5.32} that $23\lambda-35a\neq 0$. We now distinguish three cases.\par
(I) If $\lambda=0$, then \eqref{eq5.28} yields $\alpha_{1}+\alpha_{2}=0$, and it follows from
\eqref{eq5.32} and \eqref{eq5.33} that $a^{2}=\frac{81}{16}\alpha_{1}^{2}$ and
\begin{equation*}
f''(z)-\frac{2a}{9}f'(z)-\frac{a^{2}}{27}f(z)=0.
\end{equation*}
Solving the above equation, we obtain
\begin{equation}\label{eq5.34}
f(z)=d_{6}{\rm e}^{-\frac{\alpha_{1}}{4}z}+d_{7}{\rm e}^{\frac{3\alpha_{1}}{4}z},\quad
\Big(a=\frac{9\alpha_{1}}{4}\Big),
\end{equation}
or
\begin{equation}\label{eq5.35}
f(z)=d_{6}{\rm e}^{\frac{\alpha_{1}}{4}z}+d_{7}{\rm e}^{-\frac{3\alpha_{1}}{4}z},\quad
\Big(a=-\frac{9\alpha_{1}}{4}\Big),
\end{equation}
where $d_{6},~d_{7}$ are constants. Substituting \eqref{eq5.34} into \eqref{eq5.1} gives that
\begin{equation}\label{eq5.36}
\begin{split}
& \Big(\frac{3\alpha_{1}}{4}d_{7}^{4}+d_{7}q(z){\rm e}^{b+\frac{3\alpha_{1}}{4}c}\Big){\rm e}^{3\alpha_{1}z}
+\Big(2\alpha_{1}d_{6}d_{7}^{3}+d_{6}q(z){\rm e}^{b-\frac{\alpha_{1}}{4}c}\Big){\rm e}^{2\alpha_{1}z}\\
& +\frac{3\alpha_{1}}{2}d_{6}^{2}d_{7}^{2}{\rm e}^{\alpha_{1}z}-\frac{\alpha_{1}}{4}d_{6}^{4}{\rm e}^{-\alpha_{1}z}
=p_{1}{\rm e}^{\alpha_{1}z}+p_{2}{\rm e}^{-\alpha_{1}z}.
\end{split}
\end{equation}
By \eqref{eq5.36} and Lemma \ref{le2.8}, we get $d_{6}d_{7}\neq 0$,
$\frac{3\alpha_{1}}{4}d_{7}^{4}+d_{7}q(z){\rm e}^{b+\frac{3\alpha_{1}}{4}c}=0$,
$2\alpha_{1}d_{6}d_{7}^{3}+d_{6}q(z){\rm e}^{b-\frac{\alpha_{1}}{4}c}=0$,
$\frac{3\alpha_{1}}{2}d_{6}^{2}d_{7}^{2}=p_{1}$, $-\frac{\alpha_{1}}{4}d_{6}^{4}=p_{2}$.
Then $q(z)$ reduces to a constant $q$, which satisfies ${\rm e}^{\alpha_{1}c}=\frac{3}{8}$,
$-2\alpha_{1}p_{1}^{6}=243q^{4}{\rm e}^{4b}p_{2}^{3}$.
Similarly, substituting \eqref{eq5.35} into \eqref{eq5.1}, we get
${\rm e}^{\alpha_{1}c}=\frac{8}{3}$, $2\alpha_{1}p_{2}^{6}=243q^{4}{\rm e}^{4b}p_{1}^{3}$,
$\frac{\alpha_{1}}{4}d_{6}^{4}=p_{1}$, $-\frac{3\alpha_{1}}{2}d_{6}^{2}d_{7}^{2}=p_{2}$.
Then part $(2)$ of Theorem \ref{th1.3} holds.\par

(II) If $\lambda\neq 0$ and $\lambda=-2a$, then $a=-\frac{\lambda}{2}=-\frac{\alpha_{1}+\alpha_{2}}{2}$.
By \eqref{eq5.32} and \eqref{eq5.33}, we get
\begin{equation*}
f''(z)-\frac{a^{2}}{9}f(z)=0,\quad a=-\frac{\alpha_{1}+\alpha_{2}}{2}.
\end{equation*}
Solving the above equation, we get
\begin{equation}\label{eq5.37}
f(z)=d_{8}{\rm e}^{-\frac{\alpha_{1}+\alpha_{2}}{6}z}+d_{9}{\rm e}^{\frac{\alpha_{1}+\alpha_{2}}{6}z},
\end{equation}
where $d_{8},~d_{9}$ are constants. Substituting \eqref{eq5.37} into \eqref{eq5.1}, we get
\begin{equation}\label{eq5.38}
\begin{split}
& \Big[d_{8}q(z){\rm e}^{b-\frac{\alpha_{1}+\alpha_{2}}{6}c}-\frac{\alpha_{1}+\alpha_{2}}{6}d_{8}^{4}\Big]
{\rm e}^{-\frac{2(\alpha_{1}+\alpha_{2})}{3}z}+
\Big[d_{9}q(z){\rm e}^{b+\frac{\alpha_{1}+\alpha_{2}}{6}c}-\frac{\alpha_{1}+\alpha_{2}}{6}\cdot 2d_{8}^{3}d_{9}\Big]
{\rm e}^{-\frac{\alpha_{1}+\alpha_{2}}{3}z}\\
& +\frac{\alpha_{1}+\alpha_{2}}{6}\cdot2d_{8}d_{9}^{3}{\rm e}^{\frac{\alpha_{1}+\alpha_{2}}{3}z}
+\frac{\alpha_{1}+\alpha_{2}}{6}d_{9}^{4}{\rm e}^{\frac{2(\alpha_{1}+\alpha_{2})}{3}z}
=p_{1}{\rm e}^{\alpha_{1}z}+p_{2}{\rm e}^{\alpha_{2}z}.
\end{split}
\end{equation}
By \eqref{eq5.38} and Lemma \ref{le2.8}, we obtain
\begin{equation}\label{eq5.39}
\begin{split}
& d_{8}q(z){\rm e}^{b-\frac{\alpha_{1}+\alpha_{2}}{6}c}-\frac{\alpha_{1}+\alpha_{2}}{6}d_{8}^{4}=0,\quad
d_{9}q(z){\rm e}^{b+\frac{\alpha_{1}+\alpha_{2}}{6}c}-\frac{\alpha_{1}+\alpha_{2}}{6}\cdot 2d_{8}^{3}d_{9}=0,\\
& \frac{\alpha_{1}+\alpha_{2}}{6}\cdot2d_{8}d_{9}^{3}=p_{1}, \quad \frac{\alpha_{1}+\alpha_{2}}{6}d_{9}^{4}=p_{2},
\quad \frac{\alpha_{1}+\alpha_{2}}{3}=\alpha_{1}, \quad \frac{2(\alpha_{1}+\alpha_{2})}{3}=\alpha_{2},
\end{split}
\end{equation}
or
\begin{equation}\label{eq5.40}
\begin{split}
& d_{8}q(z){\rm e}^{b-\frac{\alpha_{1}+\alpha_{2}}{6}c}-\frac{\alpha_{1}+\alpha_{2}}{6}d_{8}^{4}=0,\quad
d_{9}q(z){\rm e}^{b+\frac{\alpha_{1}+\alpha_{2}}{6}c}-\frac{\alpha_{1}+\alpha_{2}}{6}\cdot 2d_{8}^{3}d_{9}=0,\\
& \frac{\alpha_{1}+\alpha_{2}}{6}\cdot2d_{8}d_{9}^{3}=p_{2},\quad \frac{\alpha_{1}+\alpha_{2}}{6}d_{9}^{4}=p_{1},
\quad \frac{\alpha_{1}+\alpha_{2}}{3}=\alpha_{2}, \quad \frac{2(\alpha_{1}+\alpha_{2})}{3}=\alpha_{1}.
\end{split}
\end{equation}
If \eqref{eq5.39} holds, then $q(z)$ reduces to a constant $q$, ${\rm e}^{\alpha_{1}c}=2$,
$\alpha_{2}=2\alpha_{1}$ and $2^{11}p_{2}^{9}q^{4}{\rm e}^{4b}=\alpha_{1}p_{1}^{12}$.
Furthermore, by \eqref{eq5.37}, we have
$f(z)=d_{8}{\rm e}^{-\frac{\alpha_{1}}{2}z}+d_{9}{\rm e}^{\frac{\alpha_{1}}{2}z}$,
$\alpha_{1}d_{8}d_{9}^{3}=p_{1}$, $\alpha_{1}d_{9}^{4}=2p_{2}$.
Similarly, if \eqref{eq5.40} holds, by \eqref{eq5.37}, we get
$f(z)=d_{8}{\rm e}^{-\frac{\alpha_{2}}{2}z}+d_{9}{\rm e}^{\frac{\alpha_{2}}{2}z}$,
$\alpha_{2}d_{8}d_{9}^{3}=p_{2}$, $\alpha_{2}d_{9}^{4}=2p_{1}$, ${\rm e}^{\alpha_{2}c}=2$,
$\alpha_{1}=2\alpha_{2}$ and $2^{11}p_{1}^{9}q^{4}{\rm e}^{4b}=\alpha_{2}p_{2}^{12}$.
Then part $(3)$ of Theorem \ref{th1.3} holds.\par

(III) If $\lambda\neq 0$ and $\lambda\neq-2a$, by \eqref{eq5.32} and \eqref{eq5.33}, we get
\begin{equation}\label{eq5.41}
f''(z)-\frac{2a+\lambda}{9}f'(z)-\frac{(\lambda-a)(\lambda-7a)(2\lambda-5a)}{27(23\lambda-35a)}f(z)=0.
\end{equation}
In what follows, we consider five subcases.\par
(III-1) If $\lambda=a$, then $a=\alpha_{1}+\alpha_{2}$. Equation \eqref{eq5.41} takes the form
\begin{equation*}
f''(z)-\frac{\alpha_{1}+\alpha_{2}}{3}f'(z)=0.
\end{equation*}
Solving the above equation yields
\begin{equation}\label{eq5.42}
f(z)=d_{10}+d_{11}{\rm e}^{\frac{\alpha_{1}+\alpha_{2}}{3}z},
\end{equation}
where $d_{10},~d_{11}$ are constants. Substituting \eqref{eq5.42} into \eqref{eq5.1} yields
\begin{equation}\label{eq5.43}
\begin{split}
& \Big[\frac{\alpha_{1}+\alpha_{2}}{3}d_{11}^{4}+d_{11}q(z){\rm e}^{b+\frac{\alpha_{1}+\alpha_{2}}{3}c}\Big]
{\rm e}^{\frac{4(\alpha_{1}+\alpha_{2})}{3}z}+
\Big[(\alpha_{1}+\alpha_{2})d_{10}d_{11}^{3}+d_{10}q(z){\rm e}^{b}\Big]{\rm e}^{(\alpha_{1}+\alpha_{2})z}\\
& +(\alpha_{1}+\alpha_{2})d_{10}^{2}d_{11}^{2}{\rm e}^{\frac{2(\alpha_{1}+\alpha_{2})}{3}z}
+\frac{\alpha_{1}+\alpha_{2}}{3}d_{10}^{3}d_{11}{\rm e}^{\frac{\alpha_{1}+\alpha_{2}}{3}z}
=p_{1}{\rm e}^{\alpha_{1}z}+p_{2}{\rm e}^{\alpha_{2}z}.
\end{split}
\end{equation}
By \eqref{eq5.43} and Lemma \ref{le2.8}, we get
\begin{equation}\label{eq5.44}
\begin{split}
& \frac{\alpha_{1}+\alpha_{2}}{3}d_{11}^{4}+d_{11}q(z){\rm e}^{b+\frac{\alpha_{1}+\alpha_{2}}{3}c}=0,\quad
(\alpha_{1}+\alpha_{2})d_{10}d_{11}^{3}+d_{10}q(z){\rm e}^{b}=0,\\
& \frac{\alpha_{1}+\alpha_{2}}{3}d_{10}^{3}d_{11}=p_{1},\quad
(\alpha_{1}+\alpha_{2})d_{10}^{2}d_{11}^{2}=p_{2},\quad
\frac{\alpha_{1}+\alpha_{2}}{3}=\alpha_{1},\quad \frac{2(\alpha_{1}+\alpha_{2})}{3}=\alpha_{2},
\end{split}
\end{equation}
or
\begin{equation}\label{eq5.45}
\begin{split}
& \frac{\alpha_{1}+\alpha_{2}}{3}d_{11}^{4}+d_{11}q(z){\rm e}^{b+\frac{\alpha_{1}+\alpha_{2}}{3}c}=0,\quad
(\alpha_{1}+\alpha_{2})d_{10}d_{11}^{3}+d_{10}q(z){\rm e}^{b}=0,\\
& \frac{\alpha_{1}+\alpha_{2}}{3}d_{10}^{3}d_{11}=p_{2},\quad
(\alpha_{1}+\alpha_{2})d_{10}^{2}d_{11}^{2}=p_{1},\quad
\frac{\alpha_{1}+\alpha_{2}}{3}=\alpha_{2},\quad \frac{2(\alpha_{1}+\alpha_{2})}{3}=\alpha_{1}.
\end{split}
\end{equation}
If \eqref{eq5.44} holds, then $q(z)$ reduces to a constant $q$, ${\rm e}^{\alpha_{1}c}=\frac{1}{3}$,
$\alpha_{2}=2\alpha_{1}$ and $3^{5}p_{1}^{6}q^{4}{\rm e}^{4b}=\alpha_{1}p_{2}^{9}$.
Furthermore, by \eqref{eq5.42}, we have
$f(z)=d_{10}+d_{11}{\rm e}^{\alpha_{1}z}$,
$\alpha_{1}d_{10}^{3}d_{11}=p_{1}$, $3\alpha_{1}d_{10}^{2}d_{11}^{2}=p_{2}$.
Similarly, if \eqref{eq5.45} holds, by \eqref{eq5.42}, we get
$f(z)=d_{10}+d_{11}{\rm e}^{\alpha_{2}z}$,
$\alpha_{2}d_{10}^{3}d_{11}=p_{2}$, $3\alpha_{2}d_{10}^{2}d_{11}^{2}=p_{1}$, ${\rm e}^{\alpha_{2}c}=\frac{1}{3}$,
$\alpha_{1}=2\alpha_{2}$ and $3^{5}p_{2}^{6}q^{4}{\rm e}^{4b}=\alpha_{2}p_{1}^{9}$.
Then part $(4)$ of Theorem \ref{th1.3} holds.\par

(III-2) If $\lambda=7a$, then $a=\frac{\alpha_{1}+\alpha_{2}}{7}$. Equation \eqref{eq5.41} reduces to
\begin{equation*}
f''(z)-\frac{\alpha_{1}+\alpha_{2}}{7}f'(z)=0.
\end{equation*}
Solving the above equation gives that
\begin{equation}\label{eq5.46}
f(z)=d_{12}+d_{13}{\rm e}^{\frac{\alpha_{1}+\alpha_{2}}{7}z},
\end{equation}
where $d_{12},~d_{13}$ are constants. Substituting \eqref{eq5.46} into \eqref{eq5.1} yields
\begin{equation}\label{eq5.47}
\begin{split}
& \Big[\frac{\alpha_{1}+\alpha_{2}}{7}d_{12}^{3}d_{13}+d_{12}q(z){\rm e}^{b}\Big]{\rm e}^{\frac{\alpha_{1}+\alpha_{2}}{7}z}
+\Big[\frac{3(\alpha_{1}+\alpha_{2})}{7}d_{12}^{2}d_{13}^{2}+d_{13}q(z){\rm e}^{b+\frac{\alpha_{1}+\alpha_{2}}{7}c}\Big]
{\rm e}^{\frac{2(\alpha_{1}+\alpha_{2})}{7}z}\\
& +\frac{3(\alpha_{1}+\alpha_{2})}{7}d_{12}d_{13}^{3}{\rm e}^{\frac{3(\alpha_{1}+\alpha_{2})}{7}z}
+\frac{\alpha_{1}+\alpha_{2}}{7}d_{13}^{4}{\rm e}^{\frac{4(\alpha_{1}+\alpha_{2})}{7}z}
=p_{1}{\rm e}^{\alpha_{1}z}+p_{2}{\rm e}^{\alpha_{2}z}.
\end{split}
\end{equation}
By \eqref{eq5.47} and Lemma \ref{le2.8}, we get
\begin{equation}\label{eq5.48}
\begin{split}
& \frac{\alpha_{1}+\alpha_{2}}{7}d_{12}^{3}d_{13}+d_{12}q(z){\rm e}^{b}=0,\quad
\frac{3(\alpha_{1}+\alpha_{2})}{7}d_{12}^{2}d_{13}^{2}+d_{13}q(z){\rm e}^{b+\frac{\alpha_{1}+\alpha_{2}}{7}c}=0,\\
& \frac{3(\alpha_{1}+\alpha_{2})}{7}d_{12}d_{13}^{3}=p_{1},\quad
\frac{\alpha_{1}+\alpha_{2}}{7}d_{13}^{4}=p_{2},\quad
\frac{3(\alpha_{1}+\alpha_{2})}{7}=\alpha_{1},\quad \frac{4(\alpha_{1}+\alpha_{2})}{7}=\alpha_{2},
\end{split}
\end{equation}
or
\begin{equation}\label{eq5.49}
\begin{split}
& \frac{\alpha_{1}+\alpha_{2}}{7}d_{12}^{3}d_{13}+d_{12}q(z){\rm e}^{b}=0,\quad
\frac{3(\alpha_{1}+\alpha_{2})}{7}d_{12}^{2}d_{13}^{2}+d_{13}q(z){\rm e}^{b+\frac{\alpha_{1}+\alpha_{2}}{7}c}=0,\\
& \frac{3(\alpha_{1}+\alpha_{2})}{7}d_{12}d_{13}^{3}=p_{2},\quad
\frac{\alpha_{1}+\alpha_{2}}{7}d_{13}^{4}=p_{1},\quad
\frac{3(\alpha_{1}+\alpha_{2})}{7}=\alpha_{2},\quad \frac{4(\alpha_{1}+\alpha_{2})}{7}=\alpha_{1}.
\end{split}
\end{equation}
If \eqref{eq5.48} holds, then $q(z)$ reduces to a constant $q$, ${\rm e}^{\alpha_{1}c}=27$,
$\alpha_{2}=\frac{4}{3}\alpha_{1}$ and $3^{9}p_{2}^{5}q^{4}{\rm e}^{4b}=\alpha_{1}p_{1}^{8}$.
Furthermore, by \eqref{eq5.46}, we have
$f(z)=d_{12}+d_{13}{\rm e}^{\frac{\alpha_{1}}{3}z}$,
$\alpha_{1}d_{12}d_{13}^{3}=p_{1}$, $\alpha_{1}d_{13}^{4}=3p_{2}$.
Similarly, if \eqref{eq5.49} holds, by \eqref{eq5.46}, we get
$f(z)=d_{12}+d_{13}{\rm e}^{\frac{\alpha_{2}}{3}z}$,
$\alpha_{2}d_{12}d_{13}^{3}=p_{2}$, $\alpha_{2}d_{13}^{4}=3p_{1}$, ${\rm e}^{\alpha_{2}c}=27$,
$\alpha_{1}=\frac{4}{3}\alpha_{2}$ and $3^{9}p_{1}^{5}q^{4}{\rm e}^{4b}=\alpha_{2}p_{2}^{8}$.
Then part $(5)$ of Theorem \ref{th1.3} holds.\par

(III-3) If $\lambda=\frac{5}{2}a$, then $a=\frac{2(\alpha_{1}+\alpha_{2})}{5}$. Equation \eqref{eq5.41} can be rewritten as
\begin{equation*}
f''(z)-\frac{\alpha_{1}+\alpha_{2}}{5}f'(z)=0.
\end{equation*}
Solving the above equation yields
\begin{equation}\label{eq5.50}
f(z)=d_{14}+d_{15}{\rm e}^{\frac{\alpha_{1}+\alpha_{2}}{5}z},
\end{equation}
where $d_{14},~d_{15}$ are constants. Substituting \eqref{eq5.50} into \eqref{eq5.1} yields
{\small
\begin{equation}\label{eq5.51}
\begin{split}
& \Big[\frac{3(\alpha_{1}+\alpha_{2})}{5}d_{14}^{2}d_{15}^{2}+d_{14}q(z){\rm e}^{b}\Big]{\rm e}^{\frac{2(\alpha_{1}+\alpha_{2})}{5}z}
+\Big[\frac{3(\alpha_{1}+\alpha_{2})}{5}d_{14}d_{15}^{3}+d_{15}q(z){\rm e}^{b+\frac{\alpha_{1}+\alpha_{2}}{5}c}\Big]
{\rm e}^{\frac{3(\alpha_{1}+\alpha_{2})}{5}z}\\
& +\frac{\alpha_{1}+\alpha_{2}}{5}d_{14}^{3}d_{15}{\rm e}^{\frac{\alpha_{1}+\alpha_{2}}{5}z}
+\frac{\alpha_{1}+\alpha_{2}}{5}d_{15}^{4}{\rm e}^{\frac{4(\alpha_{1}+\alpha_{2})}{5}z}
=p_{1}{\rm e}^{\alpha_{1}z}+p_{2}{\rm e}^{\alpha_{2}z}.
\end{split}
\end{equation}
}
By \eqref{eq5.47} and Lemma \ref{le2.8}, we get
\begin{equation}\label{eq5.52}
\begin{split}
& \frac{3(\alpha_{1}+\alpha_{2})}{5}d_{14}^{2}d_{15}^{2}+d_{14}q(z){\rm e}^{b}=0,\quad
\frac{3(\alpha_{1}+\alpha_{2})}{5}d_{14}d_{15}^{3}+d_{15}q(z){\rm e}^{b+\frac{\alpha_{1}+\alpha_{2}}{5}c}=0,\\
& \frac{\alpha_{1}+\alpha_{2}}{5}d_{14}^{3}d_{15}=p_{1},\quad
\frac{\alpha_{1}+\alpha_{2}}{5}d_{15}^{4}=p_{2},\quad
\frac{\alpha_{1}+\alpha_{2}}{5}=\alpha_{1},\quad \frac{4(\alpha_{1}+\alpha_{2})}{5}=\alpha_{2},
\end{split}
\end{equation}
or
\begin{equation}\label{eq5.53}
\begin{split}
& \frac{3(\alpha_{1}+\alpha_{2})}{5}d_{14}^{2}d_{15}^{2}+d_{14}q(z){\rm e}^{b}=0,\quad
\frac{3(\alpha_{1}+\alpha_{2})}{5}d_{14}d_{15}^{3}+d_{15}q(z){\rm e}^{b+\frac{\alpha_{1}+\alpha_{2}}{5}c}=0,\\
& \frac{\alpha_{1}+\alpha_{2}}{5}d_{14}^{3}d_{15}=p_{2},\quad
\frac{\alpha_{1}+\alpha_{2}}{5}d_{15}^{4}=p_{1},\quad
\frac{\alpha_{1}+\alpha_{2}}{5}=\alpha_{2},\quad \frac{4(\alpha_{1}+\alpha_{2})}{5}=\alpha_{1}.
\end{split}
\end{equation}
If \eqref{eq5.52} holds, then $q(z)$ reduces to a constant $q$, ${\rm e}^{\alpha_{1}c}=1$,
$\alpha_{2}=4\alpha_{1}$ and $3^{12}\alpha_{1}^{3}p_{1}^{4}p_{2}^{5}=q^{12}{\rm e}^{12b}$.
Furthermore, by \eqref{eq5.50}, we have
$f(z)=d_{14}+d_{15}{\rm e}^{\alpha_{1}z}$,
$\alpha_{1}d_{14}^{3}d_{15}=p_{1}$, $\alpha_{1}d_{15}^{4}=p_{2}$.
Similarly, if \eqref{eq5.53} holds, by \eqref{eq5.50}, we get
$f(z)=d_{14}+d_{15}{\rm e}^{\alpha_{2}z}$,
$\alpha_{2}d_{14}^{3}d_{15}=p_{2}$, $\alpha_{2}d_{15}^{4}=p_{1}$, ${\rm e}^{\alpha_{2}c}=1$,
$\alpha_{1}=4\alpha_{2}$ and $3^{12}\alpha_{2}^{3}p_{2}^{4}p_{1}^{5}=q^{12}{\rm e}^{12b}$.
Then part $(6)$ of Theorem \ref{th1.3} holds.\par

(III-4) If $\lambda \not\in\{a,~\frac{5}{2}a,~7a\}$, and Eq.\eqref{eq5.41} admits two distinct roots,
then the characteristic equation corresponding to Eq.\eqref{eq5.41} reads as follows
\begin{equation}\label{eq5.54}
\gamma^{2}-\frac{2a+\lambda}{9}\gamma-\frac{(\lambda-a)(\lambda-7a)(2\lambda-5a)}{27(23\lambda-35a)}=0.
\end{equation}
Suppose that $\gamma_{1}$ and $\gamma_{2}$ are two distinct roots of Eq.\eqref{eq5.54}. It then follows that
\begin{align}\label{eq5.55}
\left\{\begin{array}{l}
\gamma_{1}+\gamma_{2}=\frac{2a+\lambda}{9},\\
\gamma_{1}\cdot\gamma_{2}=-\frac{(\lambda-a)(\lambda-7a)(2\lambda-5a)}{27(23\lambda-35a)}.
\end{array}\right.
\end{align}
And the solution to Eq.\eqref{eq5.54} reads $f(z)=d_{16}{\rm e}^{\gamma_{1}z}+d_{17}{\rm e}^{\gamma_{2}z}$.
Substituting this expression into \eqref{eq5.1} yields
\begin{equation}\label{eq5.56}
\begin{split}
& \gamma_{1}d_{16}^{4}{\rm e}^{4\gamma_{1}z}+(3\gamma_{1}+\gamma_{2})d_{16}^{3}d_{17}{\rm e}^{(3\gamma_{1}+\gamma_{2})z}
+3(\gamma_{1}+\gamma_{2})d_{16}^{2}d_{17}^{2}{\rm e}^{2(\gamma_{1}+\gamma_{2})z}\\
& +(\gamma_{1}+3\gamma_{2})d_{16}d_{17}^{3}{\rm e}^{(\gamma_{1}+3\gamma_{2})z}
+\gamma_{2}d_{17}^{4}{\rm e}^{4\gamma_{2}z}+d_{16}{\rm e}^{b+\gamma_{1}c}q(z){\rm e}^{(a+\gamma_{1})z}\\
& +d_{17}{\rm e}^{b+\gamma_{2}c}q(z){\rm e}^{(a+\gamma_{2})z}
=p_{1}{\rm e}^{\alpha_{1}z}+p_{2}{\rm e}^{\alpha_{2}z}.
\end{split}
\end{equation}
By \eqref{eq5.56} and Lemma \ref{le2.8}, we get $3\gamma_{1}+\gamma_{2}=0$ or $\gamma_{1}+\gamma_{2}=0$ or
$\gamma_{1}+3\gamma_{2}=0$.
Next, we consider three subcases.\par
(III-4-1) If $3\gamma_{1}+\gamma_{2}=0$, then $\gamma_{2}=-3\gamma_{1}$, $a+\gamma_{2}=a-3\gamma_{1}$,
$2(\gamma_{1}+\gamma_{2})=-4\gamma_{1}$, $\gamma_{1}+3\gamma_{2}=-8\gamma_{1}$, $4\gamma_{2}=-12\gamma_{1}$.
We thus conclude that $a+\gamma_{1}$ and $a+\gamma_{2}=a-3\gamma_{1}$ must coincide with two elements from
$\{4\gamma_{1},-4\gamma_{1},-8\gamma_{1},-12\gamma_{1}\}$. Then $a+\gamma_{1}=-4\gamma_{1}$ and
$a-3\gamma_{1}=-8\gamma_{1}$ or $a+\gamma_{1}=-8\gamma_{1}$ and $a-3\gamma_{1}=-12\gamma_{1}$. If
$a+\gamma_{1}=-4\gamma_{1}$ and $a-3\gamma_{1}=-8\gamma_{1}$, then $a=-5\gamma_{1}$.
By \eqref{eq5.56} and Lemma \ref{le2.8}, we get
\begin{equation}\label{eq5.57}
\begin{split}
& 3(\gamma_{1}+\gamma_{2})d_{16}^{2}d_{17}^{2}+d_{16}{\rm e}^{b+\gamma_{1}c}q(z)=0,\quad
(\gamma_{1}+3\gamma_{2})d_{16}d_{17}^{3}+d_{17}{\rm e}^{b+\gamma_{2}c}q(z)=0,\\
& \gamma_{1}d_{16}^{4}=p_{1},\quad \gamma_{2}d_{17}^{4}=p_{2},\quad \alpha_{1}=4\gamma_{1},
\quad \alpha_{2}=-12\gamma_{1}, \quad \lambda=\alpha_{1}+\alpha_{2}=-8\gamma_{1},
\end{split}
\end{equation}
or
\begin{equation}\label{eq5.58}
\begin{split}
& 3(\gamma_{1}+\gamma_{2})d_{16}^{2}d_{17}^{2}+d_{16}{\rm e}^{b+\gamma_{1}c}q(z)=0,\quad
(\gamma_{1}+3\gamma_{2})d_{16}d_{17}^{3}+d_{17}{\rm e}^{b+\gamma_{2}c}q(z)=0,\\
& \gamma_{1}d_{16}^{4}=p_{2},\quad \gamma_{2}d_{17}^{4}=p_{1},\quad \alpha_{2}=4\gamma_{1},
\quad \alpha_{1}=-12\gamma_{1}, \quad \lambda=\alpha_{1}+\alpha_{2}=-8\gamma_{1}.
\end{split}
\end{equation}
If \eqref{eq5.57} holds, then $q(z)$ reduces to a constant $q$, $\gamma_{1}=\frac{\alpha_{1}}{4}$,
$\gamma_{2}=-3\gamma_{1}=-\frac{3\alpha_{1}}{4}$,
${\rm e}^{\alpha_{1}c}=\frac{3}{4}$,
$\alpha_{2}=-3\alpha_{1}$ and $48\alpha_{1}p_{1}p_{2}^{2}=q^{4}{\rm e}^{4b}$.
Furthermore, $f(z)=d_{16}{\rm e}^{\frac{\alpha_{1}}{4}z}+d_{17}{\rm e}^{-\frac{3\alpha_{1}}{4}z}$,
$\alpha_{1}d_{16}^{4}=4p_{1}$, $-3\alpha_{1}d_{17}^{4}=4p_{2}$,
$a=\frac{5}{8}(\alpha_{1}+\alpha_{2})$.
Similarly, if \eqref{eq5.58} holds, then we get
$f(z)=d_{16}{\rm e}^{\frac{\alpha_{2}}{4}z}+d_{17}{\rm e}^{-\frac{3\alpha_{2}}{4}z}$,
$\alpha_{1}=-3\alpha_{2}$, ${\rm e}^{\alpha_{2}c}=\frac{3}{4}$,
$\alpha_{2}d_{16}^{4}=4p_{2}$, $-3\alpha_{2}d_{17}^{4}=4p_{1}$,
$a=\frac{5}{8}(\alpha_{1}+\alpha_{2})$, and $48\alpha_{2}p_{2}p_{1}^{2}=q^{4}{\rm e}^{4b}$.
Then part $(7)$ of Theorem \ref{th1.3} holds.\par
If $a+\gamma_{1}=-8\gamma_{1}$ and $a+\gamma_{2}=a-3\gamma_{1}=-12\gamma_{1}$, then by Eq.\eqref{eq5.56}, 
we obtain $(\alpha_{1},\alpha_{2})\in \{(4\gamma_{1},-4\gamma_{1}),(-4\gamma_{1},4\gamma_{1})\}$;
that is, $\alpha_{1}+\alpha_{2}=0$, which implies that $\lambda=0$, a contradiction.\par
(III-4-2)  If $3\gamma_{2}+\gamma_{1}=0$,
part $(7)$ of Theorem \ref{th1.3} follows by the same argument.\par
(III-4-3) If $3(\gamma_{1}+\gamma_{2})=0$, then $\gamma_{1}+\gamma_{2}=0$. It follows from
\eqref{eq5.55} that $\lambda+2a=0$. This contradicts $\lambda\neq -2a$.\par
(III-5)  If $\lambda \not\in\{a,~\frac{5}{2}a,~7a\}$, and Eq.\eqref{eq5.41} possesses a repeated
root (of multiplicity two), and it associated characteristic equation is given by Eq.\eqref{eq5.54}.
Let $\beta(\neq 0)$ denote a double root of Eq.\eqref{eq5.54}. Then
\begin{align}\label{eq5.59}
\left\{\begin{array}{l}
2\beta=\frac{2a+\lambda}{9},\\
\beta^{2}=-\frac{(\lambda-a)(\lambda-7a)(2\lambda-5a)}{27(23\lambda-35a)}.
\end{array}\right.
\end{align}
And the solution to Eq.\eqref{eq5.54} reads $f(z)=(d_{18}+d_{19}z){\rm e}^{\beta z}$.
Substituting this expression into \eqref{eq5.1} gives that
{\small
\begin{align}\label{eq5.60}
(d_{18}+d_{19}z)^{3}[\beta d_{18}+(1+\beta z)d_{19}]{\rm e}^{4\beta z}+
{\rm e}^{b+\beta c}q(z)[d_{18}+d_{19}(z+c)]{\rm e}^{(a+\beta)z}
=p_{1}{\rm e}^{\alpha_{1}z}+p_{2}{\rm e}^{\alpha_{2}z}.
\end{align}
}
By \eqref{eq5.60} and Lemma \ref{le2.8}, we get
$(\alpha_{1},\alpha_{2})\in\{(4\beta,a+\beta),(a+\beta,4\beta)\}$, which implies that
$\alpha_{1}+\alpha_{2}=a+5\beta$. It follows from \eqref{eq5.28} that
$\beta=\frac{\lambda-a}{5}$. Combining this with \eqref{eq5.59}, we get $\beta=0$,
a contradiction. This completes the proof.\par

\end{proof}


\scriptsize
\begin{center}

\end{center}

\noindent Min Feng Chen\\
School of Mathematics and Statistics\\
Guangdong University of Foreign Studies\\
Guangzhou, 510006\\
P. R. China\\
Email: minfengchen2018@gdufs.edu.cn\\

\noindent Zong Xuan Chen\\
School of Mathematical Sciences\\
South China Normal University\\
Guangzhou, 510631\\
P. R. China\\
E-mail: chzx@vip.sina.com\\

\end{document}